\documentclass[letterpaper,11pt]{amsart}
\usepackage{amsmath,amsthm,amssymb,mathtools,amsfonts,mathrsfs}
\usepackage{xspace,xcolor}
\usepackage[breaklinks,colorlinks,citecolor=teal,linkcolor=teal,urlcolor=teal,pagebackref,hyperindex]{hyperref}
\usepackage[alphabetic]{amsrefs}
\usepackage{tikz-cd}
\usepackage[all]{xy}
\usepackage{bbm}

\newtheorem{thm}{Theorem}[section]
\newtheorem{prop}[thm]{Proposition}

\newtheorem{lem}[thm]{Lemma}

\theoremstyle{definition}
\newtheorem{defn}[thm]{Definition}
\newtheorem{ex}[thm]{Example}
\newtheorem{rmk}[thm]{Remark}

\newtheorem{conv}[thm]{Convention}
\newtheorem*{nota*}{Notation}

\newcommand{\C}{\mathbb C}

\newcommand{\Z}{\mathbb Z}

\newcommand{\A}{\mathbb A}
\newcommand{\Pp}{\mathbb P}

\newcommand{\sO}{\mathcal O}
\newcommand{\NE}{\operatorname{NE}}

\newcommand{\bM}{\overline{\mathcal M}}
\newcommand{\bF}{\overline{\mathcal F}}

\newcommand{\partn}{\mathcal \pi}
\newcommand{\ualpha}{{\underline\alpha}}
\newcommand{\uepsilon}{{\underline\varepsilon}}

\newcommand{\mS}{\mathfrak S}

\newcommand{\cD}{\mathcal D}
\newcommand{\HH}{\mathfrak H}
\newcommand{\fG}{{\mathfrak G}}

\newcommand{\bt}{\mathbf{t}}
\newcommand{\mt}{\mathbbm{t}}
\newcommand{\bq}{\mathbf{q}}
\newcommand{\mq}{\mathbbm{q}}
\newcommand{\bp}{\mathbf{p}}
\newcommand{\mmp}{\mathbbm{p}}
\newcommand{\smst}{\star_{\text{sm}}}
\newcommand{\g}{\mathrm{g}}
\newcommand{\ev}{\operatorname{ev}}
\newcommand{\HE}{\operatorname{HE}}
\newcommand{\rel}{\mathrm{rel}}
\newcommand{\vir}{\mathrm{vir}}
\newcommand{\Aut}{\mathrm{Aut}}

\newcommand{\GM}{\C^*}

\newcommand{\includegraphicsdpi}[3]{
    \pdfimageresolution=#1  
    \includegraphics[#2]{#3}
    \pdfimageresolution=72  
}

\title{Structures in genus-zero relative Gromov--Witten theory}

\author{Honglu Fan}
\email{honglu.fan@math.ethz.ch}
\author{Longting Wu}
\email{longting.wu@math.ethz.ch}
\author{Fenglong You}
\email{fenglong@ualberta.ca}

\begin{document}

\maketitle

\begin{abstract}
    In this paper, we define genus-zero relative Gromov--Witten invariants with negative contact orders. Using this, we construct relative quantum cohomology rings and Givental formalism. A version of Virasoro constraints also follows from it.
\end{abstract}
\tableofcontents

\section{Introduction}

\subsection{Overview}
In enumerative geometry, a modern breakthrough by Kontsevich \cite{Kon} in the 1990's showed us that the structures (quantum rings) behind curve counting problems could be the keys to solving these problems. 
With the development of Gromov--Witten theory, a lot more structural properties were discovered and generalized (Givental's quantization formalism \cite{Giv1}, Givental--Teleman's classification \cites{Giv5, Giv6, Tel}, etc.).

On the other hand, it is also natural to impose tangency conditions along a hypersurface in a counting problem. Along this idea, foundations of relative Gromov--Witten invariants were made in \cites{LR, IP2, EGH} and enjoyed further development in symplectic geometry. Later, relative Gromov--Witten invariants were also defined and studied in algebraic geometry (for example, \cites{Jun1, Jun2}, among others). Despite many years of development, parallel structures like quantum rings are still lacking in relative Gromov--Witten theory. In this paper, we propose to enlarge relative Gromov--Witten theory by allowing negative contact orders. Using this, we build structures like quantum rings and Givental formalism on relative Gromov--Witten theory.

There is a heuristic view of negative contact points. For simplicity, we assume that the target is a curve $X$. Let $D$ be the divisor corresponding to a point. Suppose that $D$ is locally defined by the equation $x=0$. Let $f:C\rightarrow X$ be a ramified cover and $p\in C$ be a ramification point over $D$. Locally at $p$, $f$ can be written as $x=z^k$ with $k>0$. ``Negative contact order" can not happen without degenerating $X$. The degeneration to the normal cone of $X$ at $D$ is locally defined by $xy=t$ where $t$ parametrizes this degeneration. At $t=0$ the curve $X$ degenerates to an $X$ glued with a $\Pp^1$ at the point $D$. Since $y$ is a local coordinate of this $\Pp^1$, a ramified cover over the $\Pp^1$ is locally described by $y=z^k$. Since $xy=t$, the local expression can be rewritten as $x=z^{-k}t$. Heuristically, a ramification point over the $\Pp^1$ at $D$ is a ``negative contact point" under the local coordinate of $X$.

It is worth noting that our version of genus-zero relative Gromov--Witten theory with negative contact orders can be completely constructed out of the original relative Gromov--Witten theory plus the Gromov--Witten theory with rubber targets (rubber theory for short). Therefore, our definition does not introduce new constructions of moduli spaces other than moduli spaces of relative stable maps. Furthermore, Remark \ref{rmk:rubcycle} suggests that in genus zero, Gromov--Witten invariants of hypersurface $D$ already completely determines rubber invariants (plus some relations about psi-classes. See \cite{MP}).

We would like to point out that the degree-zero part of the relative quantum ring of a log Calabi--Yau pair is also related to the construction of mirrors in \cites{GS, GS2} (with perhaps more discussions in a forthcoming paper by Gross, Pomerleano and Siebert \cites{GPS}). 

Following a discussion with Gross, it appears that the invariants we define should be closely related to the genus-zero punctured Gromov--Witten invariants in \cite{ACGS} (with a more complete paper coming up soon) when the boundary is a smooth divisor.

We would also like to point out that, according to \cite{GS}, relative quantum cohomology should give an algebro-geometric version of $SH^0(X\backslash D)$ (the degree-zero part of symplectic cohomology ring). In \cites{GP16, GP18},  Ganatra-Pomerleano construct a logarithmic cohomology ring $H^*_{\log}(X,\textbf{D})$ for the pair $(X,\textbf{D})$ where $\textbf{D}$ is a normal crossing divisor. They further show that $H^*_{\log}(X,\textbf{D})$ is isomorphic to $SH^*(X\backslash \textbf{D})$ under certain conditions. In the special case when $\textbf{D}=D$ is smooth, we find that only the degree-zero part of our relative quantum ring matches with that of $H^*_{\log}(X,D)$.

The keys to the structures (quantum ring and Givental formalism) are the right forms of the topological recursion relation (TRR) and the Witten--Dijkgraaf--Verlinde--Verlinde (WDVV) equation for relative Gromov--Witten invariants (see Propositions \ref{prop:TRR} ands \ref{prop:WDVV}).
Motivated by the simple relations between relative invariants and orbifold invariants of root stacks in \cites{ACW, TY}, it is natural to ask whether we are able to take TRR and WDVV from orbifold Gromov--Witten theory and pass them over to the relative theory.
Unfortunately, results of \cites{ACW, TY} are insufficient to convert TRR and WDVV into relative Gromov--Witten theory, because orbifold stable maps with large ages ($(r-k)/r$ for a fixed $k$ and a sufficiently large $r$) are crucial ingredients which are not sufficiently studied in previous works.
In this paper, those large-age markings are studied and translated into markings with negative contact orders.
We remark that a very special case of orbifold invariants with large age markings was studied in \cite{CC}*{Section 5} in details.

This paper in fact provides two equivalent definitions of relative Gromov--Witten theory with negative contact. The first definition follows from the aforementioned idea from the orbifold theory. Such a definition relies on the independence of $r$ of the orbifold theory when $r$ is sufficiently large (Theorem \ref{thm:limitexist}). The second definition is an explicit description by gluing moduli spaces of relative stable maps using fiber products. Each definition has its own merit. The first definition implies the structural properties (TRR, WDVV, etc.) directly. The second definition allows us to carry out explicit calculations and provides us with geometric insights.

There are some difficulties in higher genus. The counterexample of Maulik in \cite{ACW}*{Section 1.7} shows that the equality between genus-zero relative and orbifold invariants does not hold in higher genus. It suggests that our definition of genus-zero relative invariants with negative contact orders can not be applied to higher genus literally. However, the high-genus result in \cite{TY} suggests that we may similarly consider a suitable coefficient when $r$ tends to $\infty$. We refer to \cite{FWY2} for more details on the structures of higher genus relative Gromov--Witten theory.

\subsection{A summary of the paper}
Let $X$ be a smooth projective variety and $D$ be a smooth divisor. This paper can be outlined as follows.

\begin{enumerate}
    \item Fix a topological type $\Gamma=(g=0,n,\beta,\rho,\vec{\mu})$ of stable maps where the partition (with possibly negative entries) $\vec{\mu}=(\mu_1,\dotsc,\mu_\rho)\in (\Z^*)^\rho$ is the intersection profile with hypersurface $D$. We present two definitions of relative Gromov--Witten cycle with negative contact orders. We first define the relative Gromov--Witten cycle as the limit of the orbifold Gromov--Witten cycle $\mathfrak c_\Gamma(X/D)\in A_*(\bM_{0,n+\rho}(X,\beta)\times_{X^{\rho}} D^{\rho})$ (Definition \ref{def:cycle}). Relative Gromov--Witten invariants are defined as integrations against the Gromov--Witten cycles (Definition \ref{rel-inv-neg1}).
    
    \item We then define a relative Gromov--Witten cycle in the second way. The definition uses moduli spaces of relative stable maps and is in terms of graph sums (Definition \ref{def-rel-cycle}). Similarly, relative Gromov--Witten invariants are defined as integrations (Definition \ref{rel-inv-neg}). Some important examples are also presented in Examples \ref{ex:1neg} and \ref{ex:2neg}.
    
    \item We prove that the orbifold definition and the graph sum definition coincide (Section \ref{sec:loc}). The basic idea of the proof follows from the idea in \cites{TY18, TY}. The critical and technical part of our proof is an identity between Hurwitz--Hodge cycles and rubber cycles which is proved in the Appendix.
    
    \item We define the ring of insertions $\mathfrak H$ in Section \ref{sec:ins}. By doing so, tangency conditions are now part of the information of insertions. Thus, similar to absolute Gromov--Witten invariants, the relative invariants can be understood as a multilinear function over a suitable ring of insertions  (Definition \ref{defn:inv}).
    
    \item Now, it can be shown that the relative invariants satisfy suitable forms of TRR and WDVV (Propositions \ref{prop:TRR} and \ref{prop:WDVV}). Thus, quantum rings and Givental formalism can be defined (Sections \ref{sec:quan} and \ref{sec:Givental}). A version of Virasoro constraints is stated in Section \ref{sec:Virasoro}.
    
\end{enumerate}

Let us elaborate a little bit. Let $X_{D,r}$ be the $r$-th root stack of $X$ along the divisor $D$ for a sufficiently large integer $r$. Our first definition of relative Gromov--Witten cycles is a pushforward from the moduli of orbifold stable maps $\bM_\Gamma(X_{D,r})$ multiplied by $r^{\rho_-}$, where $\rho_-$ is the number of orbifold markings with large ages.
Our second definition of relative Gromov--Witten cycle is a pushforward from a fiber product of $\bM_{*}(X,D)$ (moduli of relative stable maps) and $\bM^\sim_{*}(D)$ (moduli of relative stable maps with rubber targets). The cycle we push forward is the natural virtual cycle of this fiber product intersecting a certain ``obstruction class" $C_{\fG}$ (see \eqref{eqn:cg}). The gluing of those moduli spaces is described in terms of a certain type of bipartite graphs. First-time readers may skip the tedious graph notation and keep only this general idea in mind. In particular, when there is only one negative contact point, the ``obstruction class" is trivial. In this case, the construction is very simple as pictured in Example \ref{ex:1neg}.

In Section \ref{sec:loc}, we show that our two definitions in Section \ref{sec:loc} are equivalent. In other words, relative invariants with negative contact orders are exactly the corresponding orbifold invariants with large ages multiplied by $r^{\rho_-}$. More precisely,
\begin{thm}[= Theorem  \ref{thm:compare}]
Fix a topological type $\Gamma=(0,n,\beta,\rho,\vec{\mu})$. For $r\gg 1$, we have the following relation for cycle classes
\begin{align}\label{equ:compare}
\mathop{\mathrm{lim}}_{r\rightarrow \infty} r^{\rho_-}\tau_*([\bM_\Gamma(X_{D,r})]^{\vir}) = \sum\limits_{\fG \in \mathcal B_\Gamma} \dfrac{1}{|\Aut(\fG)|}(\mathfrak t_{\fG})_* ({\iota}^* C_{\fG} \cap [\bM_{\fG}]^{\vir})
\end{align}
where $\bM_\Gamma(X_{D,r})$ is the moduli space of orbifold stable maps of topological type under Convention \ref{conv:Gamma}; $\rho_-$ is the number of relative markings with negative contact orders and the right-hand side of equation \eqref{equ:compare} is the virtual cycle defined via graph sums given in Section \ref{sec:rel-cycle}. In particular, the cycle classes $r^{\rho_-}\tau_*([\bM_\Gamma(X_{D,r})]^{\vir})$ are independent of $r$ when $r$ is sufficiently large. 
\end{thm}
The proof is motivated by the recent work of Tseng and the third author in \cite{TY}. Following \cite{TY}, we apply the degeneration formula (\cites{LR, Jun2}) to orbifold invariants with large ages and then apply virtual localization to the relative local model. The degeneration formula and the localization computation in Section \ref{sec:loc}, together with a key lemma in the Appendix for genus-zero Hurwitz-Hodge cycles, are finally combined to conclude the theorem. Note that the argument in this paper is on the cycle level. One can also prove a cycle-level version of \cite{TY} which in fact simplifies some of the arguments in \cite{TY}.

The idea of enlarging the ring of insertions is in fact very simple. Originally, there are two types of insertions in relative Gromov--Witten theory: the interior markings (without touching the boundary divisor $D$) and the relative markings (touching boundary divisor $D$). Insertions on these two types of markings should come from $H^*(X)$ and $H^*(D)$, respectively. Consider the direct sum 
\begin{equation}\label{eqn:vecsp}
\HH=H^*(X)\oplus\bigoplus\limits_{i\in \Z^*}H^*(D).
\end{equation}
The $H^*(X)$ piece is graded by $0$, and the other $H^*(D)$ are graded by nonzero integers. This grading stands for the contact orders at the corresponding markings, and the relative Gromov--Witten invariant is simply a multi-linear function on this enlarged vector space $\HH$.

The pairing on $\HH$ is defined as the integration of the cup product, but additionally, the contact orders need to add up to zero \eqref{eqn:pairing}. This is in fact motivated by the pairing in orbifold Gromov--Witten theory. The ring structure on $\HH$ can also be described based on a direct calculation of three-point degree-zero invariants. Interested readers may look at Section \ref{sec:ins} directly. Now a WDVV and a TRR equations can be stated and quantum rings and Givental formalism directly follow. But a subtle problem is that $\HH$ is an infinite-dimensional space. Although most of the structures (the pairing and the ring structure) on $\HH$ behave like a limit of the structures on finite-dimensional spaces, being infinite dimensional eventually becomes a problem in defining Virasoro operators (see Remark \ref{rmk:badnews}). But a version of Virasoro operators can still be stated and proven in genus zero.

\subsection{Acknowledgment}
We thank Hsian-Hua Tseng for discussions and collaborations with F. Y. on a previous project whose methods are crucial in this paper. We thank Rahul Pandharipande for discussions and suggestions. We also thank Mark Gross, Qile Chen and the anonymous referees for valuable comments on our draft. H. F. is supported by grant ERC-2012-AdG-320368-MCSK and SwissMAP. L. W. is supported by grant ERC-2017-AdG-786580-MACI. F. Y. is supported by the postdoctoral fellowship of NSERC and Department of Mathematical Sciences
at the University of Alberta. 

This project has received funding from the European Research Council (ERC) under the European Union’s
Horizon 2020 research and innovation program (grant agreement No. 786580).

\section{Relative Gromov--Witten theory}\label{sec:rel}
In this section, we would like to briefly recall relative Gromov--Witten invariants and rubber invariants. Let $X$ be a smooth projective variety and $D$ a smooth divisor. In the whole paper, the intersection number of a curve class $\beta$ with a divisor $D$ is denoted by $\int_\beta D$. 

\subsection{General theory}

Define a \emph{topological type} $\Gamma$ to be a tuple $(g,n,\beta,\rho,\vec{\mu})$ where $g,n$ are nonnegative integers, $\beta\in H_2(X,\Z)$ is a curve class and $\vec{\mu}=(\mu_1,\dotsc,\mu_\rho)\in \Z^\rho$ is a partition of the number $\int_\beta D$. Furthermore, we must have
\begin{equation}\label{eqn:positivecontact}
    \mu_i>0 \text{ for } 1\leq i\leq \rho.
\end{equation}
There is the moduli of relative stable maps $\bM_\Gamma(X,D)$ and the stabilization map $\mathfrak s:\bM_\Gamma(X,D)\rightarrow \bM_{g,n+\rho}(X,\beta)$. Write $\bar\psi_i=\mathfrak s^*\psi_i$. For $1\leq i\leq n$, it is easy to see that $\bar\psi_i$ coincides with the usual psi-class of $\bM_\Gamma(X,D)$, while for $n+1\leq i\leq n+\rho$, they are different.

There are evaluation maps
\begin{align*}
\ev_X=(\ev_{X,1},\ldots,\ev_{X,n}):&\bM_\Gamma(X,D)\rightarrow X^n, \\
\ev_D=(\ev_{D,1},\ldots,\ev_{D,\rho}):&\bM_\Gamma(X,D)\rightarrow D^\rho.
\end{align*}
The insertions of relative invariants are the following classes.
\[
\ualpha\in (\C[\bar\psi]\otimes H^*(X))^{n}, \quad \uepsilon\in (\C[\bar\psi]\otimes H^*(D))^{\rho}.
\]
For simplicity, we assume
\[
\ualpha=(\bar\psi^{a_1}\alpha_1,\ldots,\bar\psi^{a_n}\alpha_n), \quad \uepsilon=(\bar\psi^{b_1}\varepsilon_1,\ldots,\bar\psi^{b_\rho}\varepsilon_\rho).
\]

\begin{defn}\label{def:rel-inv}
The \emph{relative Gromov--Witten invariant with topological type $\Gamma$} is defined to be
\[
\langle \uepsilon \mid \ualpha \rangle_{\Gamma}^{(X,D)}=\displaystyle\int_{[\bM_{\Gamma}(X,D)]^{\vir}} \ev_D^*\uepsilon \cup \ev_X^*\ualpha,  
\]
where
\begin{equation}\label{eqn:ev}
\ev_D^*\uepsilon=\prod\limits_{j=1}^\rho \bar\psi^{b_j}_{D,j}\ev_{D,j}^*\varepsilon_j, \quad \ev_X^*\ualpha=\prod\limits_{i=1}^n \bar\psi^{a_i}_{X,i}\ev_{X,i}^*\alpha_i,
\end{equation}
with $\bar\psi_{D,j}, \bar\psi_{X,i}$ the psi-classes of the corresponding markings.
\end{defn}

We also allow disconnected domains. Let $\Gamma=\{\Gamma^\partn\}$ be a set of topological types, the relative invariant with disconnected domain curves is defined by the product rule:
\[
\langle \uepsilon\mid \ualpha\rangle_{\Gamma}^{\bullet(X,D)} = \prod\limits_{\partn} \langle \uepsilon^{\partn}\mid\ualpha^{\partn} \rangle_{\Gamma^{\partn}}^{(X,D)}.
\]
Here $\bullet$ means possibly disconnected domain and $\{\uepsilon^{\partn}\}$, $\{\ualpha^{\partn}\}$ are distributions of $\uepsilon$, $\ualpha$ according to $\Gamma^{\partn}$, respectively. We call this $\Gamma$ a \emph{disconnected topological type}. 

Also recall the definition of an {\emph {admissible graph}}.

\begin{defn}[Definition 4.6, \cite{Jun1}]
{\emph {An admissible graph}} $\Gamma$ is a graph without edges plus the following data.
\begin{enumerate}
    \item An ordered collection of legs.
    \item An ordered collection of weighted roots.
    \item A function $\g:V(\Gamma)\rightarrow \Z_{\geq 0}$.
    \item A function $b:V(\Gamma)\rightarrow H_2(X,\Z)$.
\end{enumerate}
\end{defn}
Here, we use $V(\Gamma)$ to mean the set of vertices of $\Gamma$. There is a slight difference between \cite{Jun1}*{Definition 4.6} and our definition. The function $b$ in \cite{Jun1}*{Definition 4.6} lands on $A_2(X)/\sim_{alg}$. Here we replace it by $H_2(X,\Z)$ in order to unify the notation for curve classes.

A relative stable morphism is associated with an admissible graph in the following way. Vertices in $V(\Gamma)$ correspond to the connected components of the domain curve. Roots and legs correspond to the relative markings and the interior markings, respectively. Weights on roots correspond to contact orders at the corresponding relative markings. 
The functions $\g,b$ assign a component to its genus and curve class, respectively. We do not spell out the formal definitions in order to avoid heavy notation, but we refer the readers to \cite{Jun1}*{Definition 4.7}. 
\begin{rmk}
A (disconnected) topological type and an admissible graph are equivalent concepts. 
We use the notion of admissible graphs merely due to the need of graph operations like gluing of graphs.
\end{rmk}

We use admissible graphs and topological type interchangeably in this paper. 
For the moduli space $\bM^\bullet_{\Gamma}(X,D)$ to be nonempty, we need the following extra condition on $\Gamma$ (recall $\mu_i$ are weights of roots indicating contact orders at $D$).
\begin{equation}\label{eqn:cond}
\sum\limits_{i=1}^\rho \mu_i=\int_\beta D, \qquad \mu_1,\ldots,\mu_\rho>0.
\end{equation}

\subsection{Rubber invariants}\label{sec:rub}
Given a smooth projective variety $D$ and a line bundle $L$ on $D$, we denote the \emph{moduli of relative stable maps to rubber targets} by $\bM^{\bullet\sim}_{\Gamma'}(D)$. Here $\bullet$ means possibly disconnected domain, and $\sim$ means the rubber target. The discrete data $\Gamma'$ describing the topology of relative stable maps is defined as a slight variation of the admissible graph.

\begin{defn}
A rubber graph $\Gamma'$ is an admissible graph whose roots have two different types. There are
\begin{enumerate}
    \item $0$-roots (whose weights will be denoted by $\mu^0_1,\ldots,\mu^0_{\rho_0}$), and
    \item $\infty$-roots (whose weights will be denoted by $\mu^\infty_1,\ldots,\mu^\infty_{\rho_\infty}$).
\end{enumerate}
Furthermore, the curve class assignment $b$ maps $V(\Gamma)$ to $H_2(D,\Z)$.
\end{defn}

As to the moduli space of relative stable maps to a rubber (nonrigid) target, a description can be found in \cite{GV}*{Section 2.4}. After all, a relative stable map to a rubber target of $D$ is a relative pre-stable map to a chain of $\Pp_D(L\oplus \sO)$ glued along certain invariant sections. We denote the invariant divisors at two ends of the chain by $D_0, D_\infty$. We make the convention that the normal bundles of $D_0$ and $D_\infty$ are $L$ and $L^\vee$, respectively. 

To get a non-empty moduli space, we need the following condition:
\begin{equation}\label{eqn:rubberint}
    \sum\limits_{i=1}^{\rho_0} \mu^0_i - \sum\limits_{j=1}^{\rho_\infty} \mu^\infty_j = \int_\beta c_1(L),
\end{equation}
where $\beta$ is the curve class of $\Gamma'$. If $\Gamma'$ has more than one vertex, the above is satisfied on each vertex (with $\mu^0_i,\mu^\infty_j$ corresponding to weights of roots on a given vertex).

In the rest of the paper, it is very often that a log pair $(X,D)$ is given in the context. In this case, we always assume that
\[
L=N_{D/X}.
\]
We may also refer to $N_{D/X}$ by $\sO_X(D)|_{D}$ or simply $\sO(D)$.

There is a standard way of associating a relative stable map into a rubber target with a rubber graph, where relative markings at $D_0, D_\infty$ correspond to $0$-roots and $\infty$-roots, respectively. 

We also have evaluation maps 
\[
\ev_D:\bM^{\bullet\sim}_{\Gamma'}(D)\rightarrow D^n,\quad  \ev_{D_0}:\bM^{\bullet\sim}_{\Gamma'}(D)\rightarrow D^{\rho_0},\quad \ev_{D_\infty}:\bM^{\bullet\sim}_{\Gamma'}(D)\rightarrow D^{\rho_\infty}.
\]
Given insertions 
\[\ualpha\in (\C[\bar\psi]\otimes H^*(X))^{n}, \quad \uepsilon_0\in (\C[\bar\psi]\otimes H^*(D))^{\rho_0}, \quad \uepsilon_\infty\in (\C[\bar\psi]\otimes H^*(D))^{\rho_\infty},
\]
rubber invariants are defined as follows.
\[
\langle \uepsilon_0\mid\ualpha\mid\uepsilon_\infty  \rangle^{\bullet\sim}_{\Gamma'}:=\displaystyle\int_{[\bM^{\bullet\sim}_{\Gamma'}(D)]^{\vir}} \ev_D^*\ualpha \cup \ev_{D_0}^*\uepsilon_0 \cup \ev_{D_\infty}^*\uepsilon_\infty.
\]
Similar to \eqref{eqn:ev}, we use the convention that the symbol $\bar\psi$ is changed into the corresponding $\bar\psi_i$ under $\ev_D^*, \ev_{D_0}^*, \ev_{D_\infty}^*$.

\section{Relative theory as a limit of orbifold theory}
In this section, we use orbifold Gromov--Witten theory to make our first definition of relative Gromov--Witten theory with negative contact orders. The definition in this section relies on Theorem \ref{thm:limitexist} which is proven in Section \ref{sec:loc}.

\subsection{Relative Gromov--Witten cycle with negative contact points}
Consider the $r$-th root stack $X_{D,r}$ of $X$ along the divisor $D$. Write the coarse moduli space of the inertia stack of $X_{D,r}$ as $\underline{I}(X_{D,r})$. It has $r$ components:
\[
\underline{I}(X_{D,r})  \cong X\sqcup D\sqcup D\cdots\sqcup D.
\]
The twisted sectors isomorphic to $D$ are labeled by the ages $k_i/r$, where $k_i\in\{1,2,\ldots,r-1\}$.

Since we use orbifold theory to generalize relative theory, we would like to match some of their notation. Let $\Gamma=(g,n,\beta,\rho,\vec{\mu})$ be a topological type with $\vec{\mu}=(\mu_1,\dotsc,\mu_\rho)\in (\Z^*)^\rho$ a partition of the number $\int_\beta D$. A topological type $\Gamma$ can also be used to specify topological types of orbifold stable maps to $X_{D,r}$ via the following convention. 
\begin{conv}\label{conv:Gamma}
A topological type $\Gamma=(g,n,\beta,\rho,\vec{\mu})$ of orbifold stable maps contains the following data:
\begin{itemize}
    \item $g,\beta$ correspond to the genus and curve class.
    \item $n$ indicates a set of $n$ markings without orbifold structure.
    \item $\rho$ indicates a set of $\rho$ markings with orbifold structure.
    \item $\vec{\mu}=(\mu_1,\dotsc,\mu_\rho)\in (\Z^*)^\rho$ and $\sum_{i=1}^{\rho}\mu_i=\int_\beta D$.
    \item When $\mu_i>0$, the evaluation map of the corresponding marking lands on the twisted sector with age $\mu_i/r$.
    \item when $\mu_i<0$, the evaluation map of the corresponding marking lands on the twisted sector with age $(r+\mu_i)/r$.
\end{itemize}
Here, we require that $r>\max_{1\leq i\leq \rho}|\mu_i|$.
\end{conv}

There are evaluation maps landing on the coarse moduli $\underline{I}(X_{D,r})$. Since ages are fixed by $\vec{\mu}$, we further restrict their targets to the corresponding components. We have the restricted evaluation maps
\[
\ev_X:\bM_\Gamma(X_{D,r})\rightarrow X^n, \quad \ev_D:\bM_\Gamma(X_{D,r})\rightarrow D^\rho
\]
corresponding to those $n$ markings without orbifold structures, and those $\rho$ markings with orbifold structures, respectively. We similarly denote their entries by
\[
\ev_X = (\ev_{X,1},\ldots,\ev_{X,n}), \quad \ev_D = (\ev_{D,1},\ldots,\ev_{D,\rho}).
\]
Consider the forgetful map
\[
\tau: \bM_\Gamma(X_{D,r})\rightarrow \bM_{0,n+\rho}(X,\beta)\times_{X^{\rho}} D^{\rho},
\]
we write $\bar{\psi_i}=\tau^*\psi_i$.
Using the notation in Definition \ref{def:rel-inv}, the \emph{orbifold Gromov--Witten invariant with topological type $\Gamma$} is
\begin{equation}\label{orb-inv}
    \langle \uepsilon, \ualpha \rangle_{\Gamma}^{X_{D,r}}=\displaystyle\int_{[\bM_{\Gamma}(X_{D,r})]^{\vir}} \prod\limits_{j=1}^\rho \bar\psi_{D,j}^{b_j}\ev_{D,j}^*\varepsilon_j\prod\limits_{i=1}^n \bar\psi_{X,i}^{a_i}\ev_{X,i}^*\alpha_i,
\end{equation}
where $\bar\psi_{D,j}, \bar\psi_{X,i}$ are psi-classes corresponding to markings evaluated under $\ev_D, \ev_X$.

For orbifold Gromov--Witten invariants (\ref{orb-inv}) with topological type $\Gamma$, we define an integer $\rho_-\in \mathbb Z_{\geq 0}$ to be
\[
\rho_-=\sum_{\mu_i>0}\mu_i/r+\sum_{\mu_i<0}(r+\mu_i)/r - \left( \int_\beta D \right)/r=\#\{\mu_i<0\}.
\]

Before making the definition, we need to state a fact on which our definition relies. Under the map $\tau$, the $\rho$ relative markings become ordinary markings in $\bM_{0,n+\rho}(X,\beta)$. Let $\bM_{0,n+\rho}(X,\beta)\times_{X^{\rho}} D^{\rho}$ be the fiber product of the evaluation map of those $\rho$ markings $\bM_{0,n+\rho}(X,\beta)\rightarrow X^\rho$ and the embedding $D^\rho\hookrightarrow X^{\rho}$. 
\begin{thm}\label{thm:limitexist}
Fix a topological type $\Gamma=(0,n,\beta,\rho,\vec{\mu})$. For sufficiently large $r\gg 1$, the following cycle class
\[
r^{\rho_-}\tau_*([\bM_\Gamma(X_{D,r})]^{\vir})\in A_*(\bM_{0,n+\rho}(X,\beta)\times_{X^{\rho}} D^{\rho})
\]
is independent of $r$.
\end{thm}
Theorem \ref{thm:limitexist} will be proven in Section \ref{sec:loc}. For sufficiently large $r$, denote the above cycle by \[\mathop{\mathrm{lim}}_{r\rightarrow \infty} r^{\rho_-}\tau_*([\bM_\Gamma(X_{D,r})]^{\vir}).\]

\begin{defn}\label{def:cycle}
Fix a topological type $\Gamma=(0,n,\beta,\rho,\vec{\mu})$. Define the \emph{relative Gromov--Witten cycle} of topological type $\Gamma$ as
\[
\mathfrak c_\Gamma(X/D) = \mathop{\mathrm{lim}}_{r\rightarrow \infty}r^{\rho_-}\tau_*([\bM_\Gamma(X_{D,r})]^{\vir}) \in A_*(\bM_{0,n+\rho}(X,\beta)\times_{X^{\rho}} D^{\rho}).
\]
\end{defn}

When $\rho_-=0$, $\mathfrak c_\Gamma(X/D)$ coincides with the pushforward of virtual cycle from the corresponding moduli of relative stable maps according to \cite{ACW}.

Recall that $\rho$ is used to denote the number of roots in the admissible graph $\Gamma$. Let $\rho_+$ be the number of roots whose weights are positive. We have
\begin{prop}\label{prop:vdim}
\[\mathfrak c_\Gamma(X/D) \in A_{d}(\bM_{0,n+\rho}(X,\beta)\times_{X^{\rho}} D^{\rho}),\] where 
\[d=\mathrm{dim}(X)-3+\int_{\beta} c_1(T_X(-\mathrm{log} D)) + n + \rho_+. \]
\end{prop}
\begin{proof}
By Riemann-Roch theorem, we know that the virtual dimension of $\bM_\Gamma(X_{D,r})$ is given by 
\begin{eqnarray*}
& & \int_\beta c_1(T_{X_{D,r}})+(\dim (X)-3)(1-0)+n+\rho-\sum_{i:\mu_i> 0}\frac{\mu_i}{r}-\sum_{i:\mu_i<0}\frac{r+\mu_i}{r} \\
& = & \dim(X)-3+\int_{\beta} c_1(T_X(-\mathrm{log} D))+\frac{\int_\beta D}{r}-\frac{\sum_i\mu_i}{r}+n+\rho_+\\
& = & \mathrm{dim}(X)-3+\int_{\beta} c_1(T_X(-\mathrm{log} D)) + n + \rho_+.
\end{eqnarray*}
\end{proof}

\subsection{Relative invariants with negative contact orders}\label{rel-inv1}
We define relative Gromov--Witten invariants (possibly with negative contact orders) by integrations against this cycle. 

Let 
\begin{align*}
    \begin{split}
        \ualpha = (\bar\psi^{a_1}\alpha_1,\ldots,\bar\psi^{a_n}\alpha_n) 
        &\in (\C[\bar\psi] \otimes H^*(X))^{n},\\ \uepsilon = (\bar\psi^{b_1}\epsilon_1,\ldots,\bar\psi^{b_{\rho}}\epsilon_\rho) &\in (\C[\bar\psi] \otimes H^*(D))^{\rho}.
    \end{split}
\end{align*}
There are evaluation maps from $\bM_\Gamma(X_{D,r})$ corresponding to interior markings and relative markings:
\begin{align*}
\ev_X=(\ev_{X,1},\ldots,\ev_{X,n}):\bM_\Gamma(X_{D,r})&\rightarrow X^n, \\
\ev_D=(\ev_{D,1},\ldots,\ev_{D,\rho}):\bM_\Gamma(X_{D,r})&\rightarrow D^\rho.
\end{align*}
There are also evaluation maps
\begin{align*}
\overline{\ev}_X=(\overline\ev_{X,1},\ldots,\overline\ev_{X,n}):\bM_{0,n+\rho}(X,\beta)\times_{X^{\rho}} D^{\rho}&\rightarrow X^n, \\
\overline{\ev}_D=(\overline\ev_{D,1},\ldots,\overline\ev_{D,\rho}):\bM_{0,n+\rho}(X,\beta)\times_{X^{\rho}} D^{\rho}&\rightarrow D^\rho,
\end{align*}
such that 
\[
\overline{\ev}_X\circ \tau=\ev_X, \quad \overline{\ev}_D\circ\tau=\ev_D.
\]

\begin{defn}\label{rel-inv-neg1}
The {\it relative Gromov--Witten invariant of topological type $\Gamma$ with insertions $\uepsilon,\ualpha$} is
\[
\langle \uepsilon \mid \ualpha \rangle_{\Gamma}^{(X,D)} =  \displaystyle\int_{\mathfrak c_\Gamma(X/D)} \prod\limits_{j=1}^{\rho} \bar{\psi}_{D,j}^{b_j}\overline{\ev}_{D,j}^*\epsilon_j\prod\limits_{i=1}^n \bar{\psi}_{X,i}^{a_i}\overline{\ev}_{X,i}^*\alpha_i,
\]
where $\bar\psi_{D,j}, \bar\psi_{X,i}$ are pullback of psi-classes from $\bM_{0,n+\rho}(X,\beta)$ to $\bM_{0,n+\rho}(X,\beta)\times_{X^{\rho}} D^{\rho}$ corresponding to markings evaluated under $\overline\ev_D,\overline\ev_X$.
\end{defn}

\section{Graph notation}\label{sec:graph}
In Sections \ref{sec:graph} and \ref{sec:rel-neg}, the second definition of relative Gromov--Witten theory with negative contact orders will be given. In this new definition, all ingredients are in terms of relative moduli and rubber moduli in the sense of \cites{Jun1,Jun2}. The second definition has some geometrical and computational benefits. The fact that the two definitions coincide will be established in Section \ref{sec:loc}.

The purpose of this section is to establish the notation for a special type of decorated bipartite graph. The bipartite graph has two sides which we label as $0$-side and $\infty$-side. We first define graphs with half-edges on each side (graphs of type $0$ and graphs of type $\infty$), and then glue the corresponding half-edges in a specific way. Before throwing out formal definitions, maybe it is helpful to provide some geometric explanations. 

Geometrically, $0$-side corresponds to rubber targets over $D$, and $\infty$-side corresponds to $X$. Rubbers over $0$-side contain a boundary divisor $D$ with prescribed contact orders (corresponding to $0$-roots on it). The other end glues with $X$ on the $\infty$-side along the divisor $D\subset X$. What is different from the picture of \cite{Jun1} is that the ``crease" between rubbers and $X$ (the invariant section of rubbers that glues to $X$) may contain markings (corresponding to $\infty$-roots of marking type) that do not form a node with balancing contact orders. Furthermore, since there can be multiple vertices over $0$-side, we do allow several independent rubber targets to glue to $X$ along the same $D$. 

\subsection{Graphs of type $0$ and type $\infty$}
We first define graphs with half-edges at $0$-side and $\infty$-side.
\begin{defn}\label{def:admgraph0}
\emph{A (connected) graph of type $0$} is a weighted graph $\Gamma^0$ consisting of a
single vertex, no edges, and the following four types of half-edges.
\begin{enumerate}
\item {$0$-roots},
\item {$\infty$-roots of node type},
\item {$\infty$-roots of marking type},
\item {Legs}.
\end{enumerate}
$0$-roots are weighted by
positive integers, and $\infty$-roots are weighted by negative integers. The vertex is associated with a tuple $(g,\beta)$ where $g\geq 0$
and $\beta\in H_2(D,\Z)$. 
\end{defn}

\begin{ex}\label{ex:1}
Omitting the decoration $(g,\beta)$ and all the weights, the following is a valid graph of type $0$.

\includegraphicsdpi{600}{}{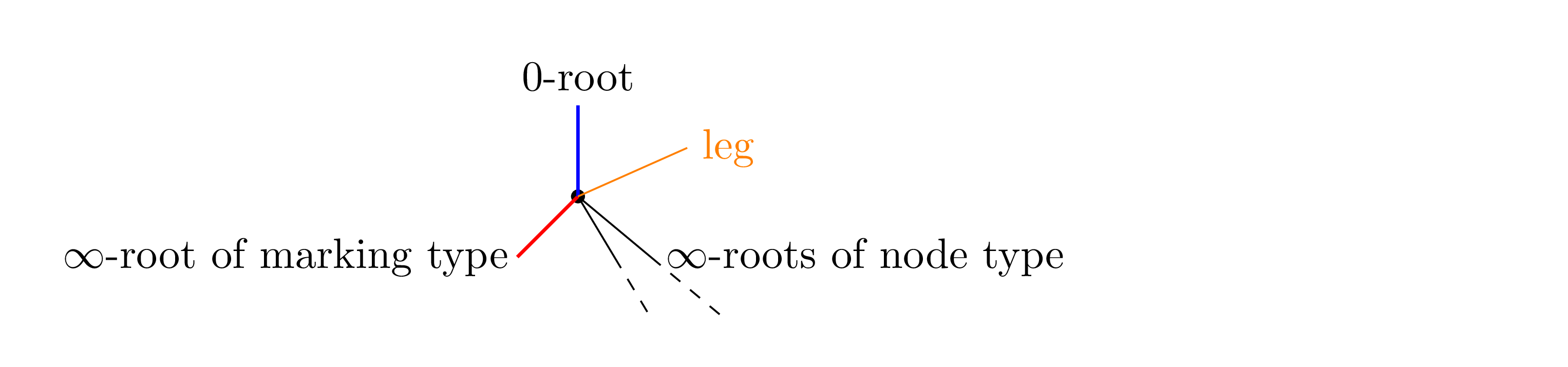}
Dashed lines on $\infty$-roots of node type indicate that these half edges will form full edges by gluing. In this picture, negatively weighted roots are pointing downwards ($\infty$-roots of both types).
\end{ex}

\begin{defn}\label{def:halfedges}
We denote 
\begin{itemize}
    \item the set of legs by $\HE_{l}(\Gamma^0)$,
    \item the set of $0$-roots by $\HE_0(\Gamma^0)$,
    \item the set of $\infty$-roots of marking type by $\HE_{m}(\Gamma^0)$,
    \item and the set of $\infty$-roots of node type by $\HE_{n}(\Gamma^0)$.
\end{itemize}
To simplify our notation later, we also define the following.
\begin{itemize}
    \item $\HE_{l,0}(\Gamma^0)=\HE_{l}(\Gamma^0)\coprod \HE_{0}(\Gamma^0)$.
    \item $\HE_{m,n}(\Gamma^0)=\HE_{m}(\Gamma^0)\coprod \HE_{n}(\Gamma^0)$.
    \item $\HE_{l,0,m}(\Gamma^0)=\HE_{l}(\Gamma^0)\coprod \HE_{0}(\Gamma^0)\coprod\HE_{m}(\Gamma^0)$.
    \item $\HE(\Gamma^0)=\HE_{l}(\Gamma^0)\coprod \HE_{0}(\Gamma^0)\coprod\HE_{m}(\Gamma^0)\coprod\HE_{n}(\Gamma^0)$ (the set of all half edges).
\end{itemize} 
\end{defn}

Later, graphs of type $0$ will be used in two different ways. One is to record the topological type of an orbifold stable map to the $r$th root gerbe $\sqrt[r]{D/L}$ over $D$ with respect to a line bundle $L$. And the other one is to record the topological type of a relative stable map to a rubber target over $D$. In order to match the notation with Section \ref{sec:localmodel}, we write
\[
\cD_0=\sqrt[r]{D/L}.
\]

\begin{defn}
Define $\bM_{\Gamma^0}(\cD_0)$ to be the moduli stack of genus $g$, degree $\beta$ orbifold stable maps to $\cD_0$ whose markings correspond to half-edges of $\Gamma^0$ with the following assignments of ages:
A $0$-root of weight
$i$ corresponds to a marking of age $i/r$. An $\infty$-root (of either type) of weight $i$ corresponds to a
marking of age $(r+i)/r$ (recall $i$ is negative in this case). A leg corresponds to a marking of age $0$ (that is, without orbifold structure).
\end{defn}

Recall that in the relative theory with rubber (nonrigid) target over $D$, the target expands as a chain of $\Pp_D(L\oplus \sO)$ glued along suitable sections. Two distinguished sections are denoted by $D_0$ and $D_\infty$. Our convention is that the normal bundle of $D_0$ is $L$.
\begin{defn}\label{def:typ0}
Define $\bM^\sim_{\Gamma^0}(D)$ to be the moduli stack of genus $g$, degree $\beta$ relative stable maps to rubber targets over $D$. Each marking corresponds to a half-edge of $\Gamma^0$ with the following assignments of contact orders: a $0$-root of weight $i$ corresponds to a relative marking over $D_0$ with contact order $i$. An $\infty$-root (of either type) of weight $i$ corresponds to a relative marking over $D_\infty$ of contact order $-i$ (recall $i$ is negative). A leg corresponds to an interior marking.
\end{defn}

On the other hand, the term \emph{graph of type $\infty$} is simply an admissible graph such that the roots are distinguished by node type and marking type. The term ``type $\infty$" indicates its position in the bipartite graph introduced later. 

\begin{ex}\label{ex:2}
Similar to Example \ref{ex:1}, we omit the decoration and weights. The following is a valid graph of type $\infty$.

\includegraphicsdpi{600}{}{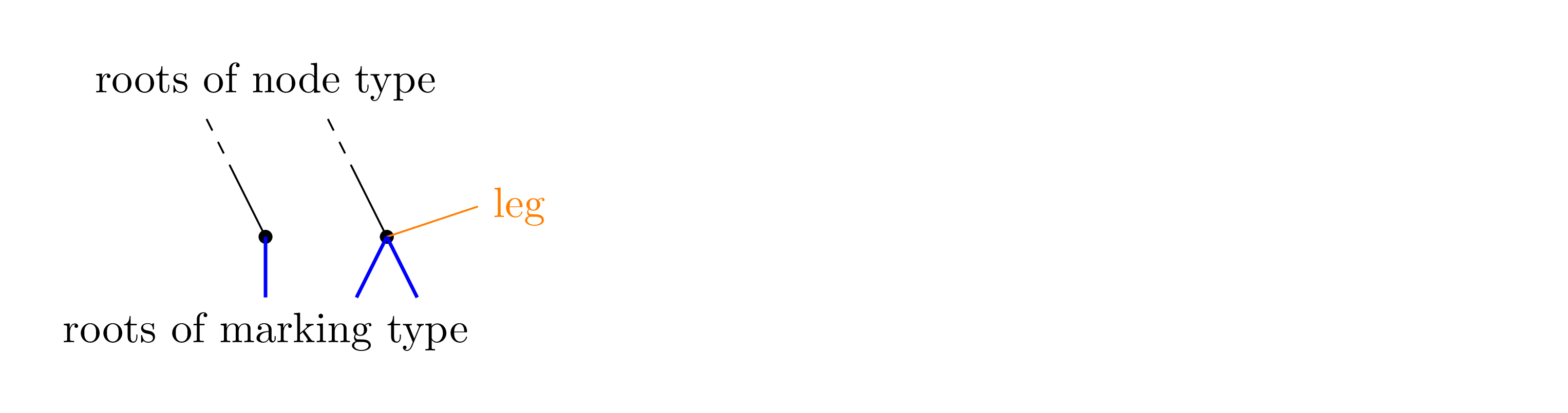}
Note that a graph of type $0$ must consist of one single vertex, but a graph of type $\infty$ may be disconnected. All roots are weighted positively.
\end{ex}

\begin{defn}\label{def:halfedges-new} Let $\Gamma^{\infty}$ be a graph of type $\infty$. We define
\begin{itemize}
    \item $\HE_{l}(\Gamma^{\infty})$ to be the set of legs,
    \item $\HE_{m}(\Gamma^{\infty})$ to be the set of roots of marking type,
    \item and $\HE_{l,m}(\Gamma^{\infty})=\HE_{l}(\Gamma^{\infty})\coprod\HE_{m}(\Gamma^{\infty}).$
\end{itemize}
\end{defn}

\subsection{Admissible and localization bipartite graphs}
We use graphs of type $0$ and $\infty$ to define a type of decorated bipartite graphs.
\begin{defn}\label{defn:locgraph}
\emph{An admissible bipartite graph} $\fG$ is a tuple
$(\mS_0,\Gamma^\infty,I,E,\g,b)$, where each element is explained as follows.
\begin{enumerate}
\item \emph{(Vertices)} $\mS_0=\{\Gamma_i^0\}$ is a set of graphs of type $0$;
  $\Gamma^\infty$ is a (possibly disconnected) graph of type $\infty$.
\item \emph{(Edges)} $E$ is a set of pairs $((l,\Gamma^0_i),(l',\Gamma^\infty))$, where $l$ is an $\infty$-root of node type in $\Gamma^0_i\in \mS_0$, and $l'$ is a root  
of node type in $\Gamma^\infty$.
\item \emph{(Markings)} $I$ is a one-to-one correspondence between the set $\{1,\ldots, n+\rho\}$ and the set $\coprod\limits_{\Gamma_i^0\in \mS_0}\HE_{l,0,m}(\Gamma_i^0)\coprod \HE_{l,m}(\Gamma^{\infty})$ such that
\begin{itemize}
    \item $\coprod\limits_{\Gamma_i^0\in \mS_0}\HE_{l}(\Gamma_i^0)\coprod \HE_l(\Gamma^{\infty})$ corresponds to the subset $\{1,\ldots,n\}$,
    \item and $\coprod\limits_{\Gamma_i^0\in \mS_0}\HE_{0,m}(\Gamma_i^0)\coprod \HE_m(\Gamma^{\infty})$ corresponds to the subset $\{n+1,\ldots,n+\rho\}$
\end{itemize}
  (for the notation $\HE_{*}(\cdot)$, see Definition \ref{def:halfedges} and \ref{def:halfedges-new}).
\item \emph{(Genus and degree)} \[\g:\bigcup\limits_{\Gamma^0_i\in \mS_0} V(\Gamma^0_i)\cup V(\Gamma^\infty)\rightarrow \Z_{\geq 0},\]
  \[b:\bigcup\limits_{\Gamma^0_i\in \mS_0} V(\Gamma^0_i)\cup V(\Gamma^\infty)\rightarrow H_2(D,\Z)\cup H_2(X,\Z)\] are maps between sets such that
  \[
  b\bigg(\bigcup\limits_{\Gamma^0_i\in \mS_0} V(\Gamma^0_i)\bigg)\subset H_2(D,\Z), \qquad b\bigg(V(\Gamma^\infty)\bigg)\subset H_2(X,\Z).
  \]
\end{enumerate}
In addition, $\Gamma$ satisfies the following.
\begin{enumerate}
    \item Each $\infty$-root of node type in admissible graphs of $\mS_0$, and each root of node type in $\Gamma^\infty$ appears exactly once as an element of a pair in $E$.
    \item The maps $\g,b$ are compatible with the genus and degree decorations on admissible graphs.
    \item For a graph $\Gamma^0_i$, the sum of weights of all roots equals $\int_{b(\Gamma^0_i)} D$. For a vertex in $\Gamma^\infty$, the sum of weights of all roots also equals the intersection of curve class with $D$.
    \item For each edge $((l,\Gamma^0_i),(l',\Gamma^\infty))$, the weights of $l$ and $l'$ add up to $0$.
    \item All the vertices in $\bigcup\limits_{\Gamma^0_i\in \mS_0} V(\Gamma^0_i)\cup V(\Gamma^\infty)$ are \emph{stable}. We call a vertex $v$ \emph{stable} if either $b(v)\neq 0$ or the number of half-edges associated with $v$ is bigger than $2-2g(v)$.
\end{enumerate}
\end{defn}

In the definition, $E$ should be understood as edges connecting type $0$ and type $\infty$ graphs at two ends. We say $\fG$ is connected if the actual bipartite graph is a connected graph. An automorphism of a bipartite graph is defined as a permutation of vertices and edges that preserves legs, roots of marking type and the assignments of genus and degree. The group of automorphisms is denoted by $\operatorname{Aut}(\fG)$.

\begin{ex}\label{ex:3}
Suppose that $\mS_0$ consists of a single graph as in Example \ref{ex:1}, and $\Gamma^\infty$ is the graph in Example \ref{ex:2}. By joining $\infty$-roots of node type in $\Gamma^0$ with roots of node type in $\Gamma^\infty$ (assume condition (d) above is satisfied), we obtain the following admissible bipartite graph.

\includegraphicsdpi{600}{}{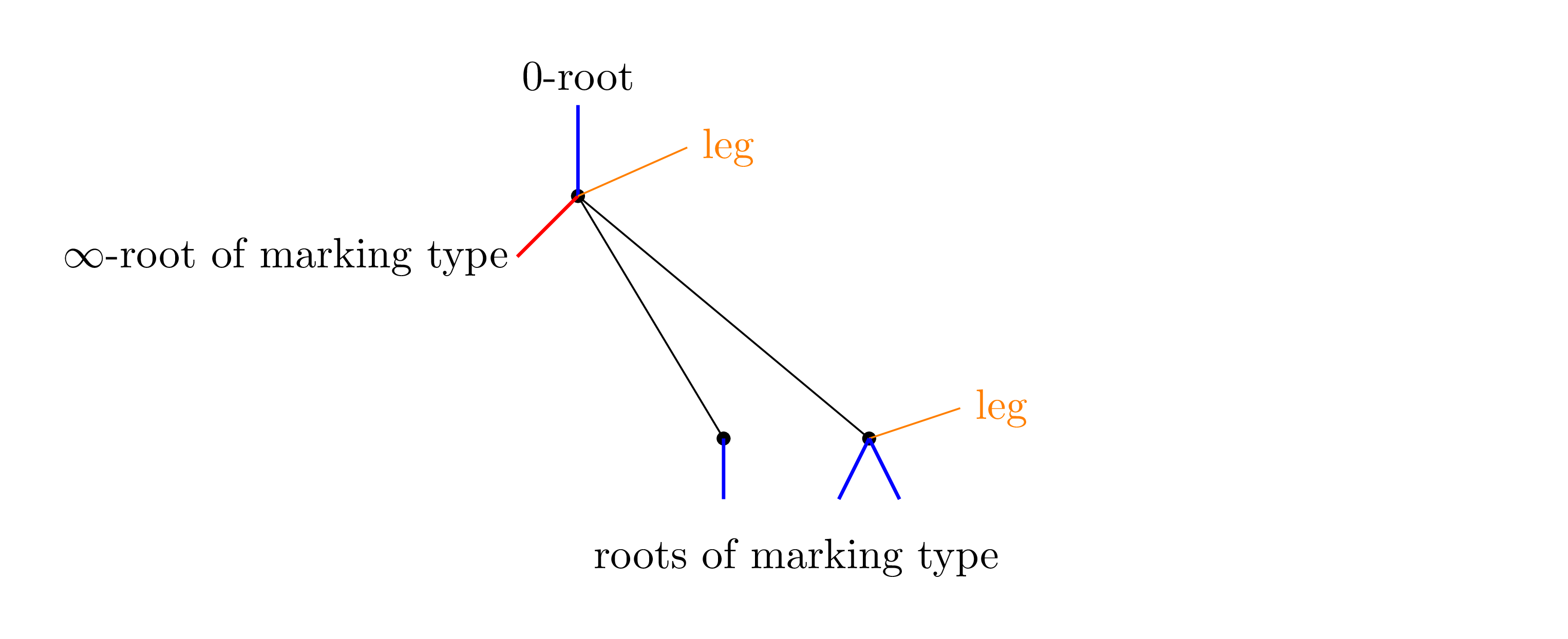}
After gluing, $\infty$-roots of marking type are the only negatively weighted roots (and ultimately corresponding to negative contact points). Roots of node type now form edges. We suggest readers to think of the absolute values of their weights as ``multiplicities of the edges".
\end{ex}

For technical purposes, we need to define a variation of admissible bipartite graphs.
\begin{defn}\label{def:loc-graph}
\emph{A localization bipartite graph} $\fG$ is a tuple $(\mS_0,\Gamma^\infty,I,E,\g,b)$, where all the definitions and requirements are the same as an admissible bipartite graph except the following.
\begin{itemize}
    \item $\Gamma^\infty$ is a rubber graph. The rubber graph $\Gamma^\infty$ satisfies constraint \eqref{eqn:rubberint}.
    \item Since a rubber graph $\Gamma^\infty$ no longer divides roots into node type and marking type, any $0$-root in $\Gamma^\infty$ is connected to an $\infty$-root of node type in $\Gamma^0_i$ by some edge in $E$.
    \item The curve class assignment $b$ has a different (smaller) target. More precisely, $b$ is now a map
    \[b:\bigcup\limits_{\Gamma^0_i\in \mS_0} V(\Gamma^0_i)\cup V(\Gamma^\infty)\rightarrow H_2(D,\Z),\]
    that is, $b$ maps both $V(\Gamma^0_i)$ and $V(\Gamma^\infty)$ into $H_2(D,\Z)$.
    \item $\Gamma^0_i$ and $\Gamma^\infty$ are allowed to be unstable. We call $\Gamma^0_i$ (respectively, $\Gamma^\infty$) unstable if every vertex in $V(\Gamma^0_i)$ (respectively, $V(\Gamma^\infty)$) is unstable.
\end{itemize} 
\end{defn}
As the name indicates, a localization bipartite graph is used in the localization process in Section \ref{sec:loc}. If the reader skips the technical details in Section \ref{sec:loc}, this definition can be safely ignored. 

\begin{rmk}\label{rmk:unstable}
If $\Gamma^\infty$ contains an unstable vertex which is connected to another vertex in some $\Gamma^0_i$, it would be of genus zero, curve class $0$ with one $0$-root, one $\infty$-root and nothing else. In Section \ref{sec:loc}, the case that $\Gamma^{\infty}$ is unstable will correspond to the fixed locus where the target does not degenerate over $\infty$.
\end{rmk}

Given an admissible bipartite graph $\fG$, we can talk about its topological type. More precisely, we make the following definition.
\begin{defn}\label{def:topotype}
The topological type of an admissible bipartite graph $\fG$ is a tuple $(g,n,\beta,\rho,\vec{\mu})$ where 
\begin{itemize}
    \item $g$ is the sum of genera (values of $\g$) of all vertices plus $h^1(\fG)$ (as $1$-dimensional CW-complex);
    \item $\beta\in H_2(X,\Z)$ is the sum of curve classes (values of $b$) of all vertices (curve classes of $\Gamma^0_i$ in $H_2(D,\Z)$ are pushed forward to $H_2(X,\Z)$);
    \item $n$ is the number of legs;
    \item $\rho$ is the total number of $0$-roots, $\infty$-roots of marking type in $\mS_0$ and roots of marking type in $\Gamma^{\infty}$;
    \item and $\vec{\mu}$ is the list of weights of $0$-roots, $\infty$-roots of marking type in $\mS_0$ and roots of marking type in $\Gamma^{\infty}$.
\end{itemize}   
\end{defn}
When $\fG$ is a localization bipartite graph, we define its topological type by $(g,n,\beta,\rho,\vec{\mu},\rho_{\eta},\vec{\eta})$, where $g,n$ are similar to Definition \ref{def:topotype} plus the following.
\begin{itemize}
    \item $\beta$ is the sum of curve classes as well, but lies in $H_2(D,\Z)$;
    \item $\rho$ is the total number of $0$-roots, $\infty$-roots of marking type in $\mS_0$ and $0$-roots of marking type in $\Gamma^{\infty}$;
    \item $\vec{\mu}$ is the list of weights of $0$-roots, $\infty$-roots of marking type in $\mS_0$ and $0$-roots of marking type in $\Gamma^{\infty}$;
    \item In addition, $\rho_\eta$ is the number of $\infty$-roots in $\Gamma^\infty$. And $\vec{\eta}$ is the list of weights of $\infty$-roots in $\Gamma^\infty$.
\end{itemize}

\section{Relative theory as a graph sum}\label{sec:rel-neg}
From now on, we focus on admissible bipartite graphs with $g=0$. 
We are ready to state the second definition of relative Gromov--Witten cycles (with negative contact orders) in this section. 

\subsection{The set-up and notation}\label{sec:rel-neg-set-up}
Let $X$ be a smooth projective variety and $D$ a smooth divisor. Later in computations, we also use $D$ for the class $c_1(\sO_X(D))$ in $H^2(X)$. Its restriction to the hypersurface $D$ is also frequently used in insertions at relative markings. Because the context is clear, we abuse the notation by using the same $D$ for $c_1(\sO_X(D))|_D\in H^2(D)$. We also remark that $D$ as the divisor class is the same as the first Chern class of $N_{D/X}$.

Let $\Gamma=(0,n,\beta,\rho,\vec{\mu})$ with $\vec{\mu}=(\mu_1,\dotsc,\mu_\rho)\in (\Z^*)^\rho$ satisfying 
\[
\sum\limits_{i=1}^\rho \mu_i = \int_\beta D. 
\]
Let $\mathcal B_\Gamma$ be the set of connected admissible bipartite graphs of topological type $\Gamma$. 

Given a bipartite graph $\fG\in \mathcal B_\Gamma$, 
\begin{equation}\label{eqn:fiberprod}
\bM_{\fG} = \prod\limits_{\Gamma^0_i\in \mS_0}\bM^\sim_{\Gamma^0_i}(D) \times_{D^{|E|}} \bM^{\bullet}_{\Gamma^\infty}(X,D),
\end{equation}
where $\times_{D^{|E|}}$ is the fiber product identifying evaluation maps according to edges (specified in the set $E$). For $\bM^\sim_{\Gamma_i^0}(D)$, see Definition \ref{def:typ0}. For $\bM^{\bullet}_{\Gamma^\infty}(X,D)$, recall that it is the moduli of relative stable maps with possibly disconnected domains of type $\Gamma^{\infty}$. In particular, we have the following diagram.
\begin{equation}\label{eqn:diag}
\xymatrix{
\bM_{\fG} \ar[r]^{} \ar[d]^{\iota} & D^{|E|} \ar[d]^{\Delta} \\
\prod\limits_{\Gamma^0_i\in \mS_0}\bM^\sim_{\Gamma^0_i}(D) \times \bM^{\bullet}_{\Gamma^\infty}(X,D) \ar[r]^{} & D^{|E|}\times D^{|E|}.
}
\end{equation}
There is a natural virtual class $[\bM_{\fG}]^{\vir}$. In fact, we have
\[
[\bM_{\fG}]^{\vir}=\Delta^![\prod\limits_{\Gamma^0_i\in \mS_0}\bM^\sim_{\Gamma^0_i}(D) \times \bM^{\bullet}_{\Gamma^\infty}(X,D)]^{\vir}
\]
where $\Delta^!$ is the Gysin map.

For each $\bM^\sim_{\Gamma_i^0}(D)$, we have a stabilization map $\bM^\sim_{\Gamma_i^0}(D) \rightarrow \bM_{0,n_i+\rho_i}(D,\beta_i)$ where $n_i$ is the number of legs, $\rho_i$ is the number of $0$-roots plus the number of $\infty$-roots of marking type, and $\beta_i$ is the curve class of $\Gamma_i^0$. As a result, there is a map \[
\bM_{\fG}=\prod\limits_{\Gamma^0_i\in \mS_0}\bM^\sim_{\Gamma^0_i}(D) \times_{D^{|E|}} \bM^{\bullet}_{\Gamma^\infty}(X,D)
\rightarrow 
\prod\limits_{\Gamma^0_i\in \mS_0}\bM_{0,n_i+\rho_i}(D,\beta_i) \times_{D^{|E|}} \bM^{\bullet}_{\Gamma^\infty}(X,D).
\]

On the other hand, there is a boundary map
\[
\prod\limits_{\Gamma^0_i\in \mS_0}\bM_{0,n_i+\rho_i}(D,\beta_i) \times_{D^{|E|}} \bM^{\bullet}_{\Gamma^\infty}(X,D) \rightarrow \bM_{0,n+\rho}(X,\beta)\times_{X^{\rho}} D^{\rho}
\]
gluing curves according to $E$ (note that $\fG$ is connected), and label markings according to $I$. 
By composing these two, we obtain a map 
\[
\mathfrak t_{\fG}:\bM_{\fG}\rightarrow \bM_{0,n+\rho}(X,\beta)\times_{X^{\rho}} D^{\rho}.
\]

\begin{ex}
The fiber product construction corresponds to a gluing process of curves. We demonstrate this gluing in this example. Let $\fG$ be the admissible bipartite graph in Example \ref{ex:3}. The following illustration shows the corresponding curve type: 

\includegraphicsdpi{600}{}{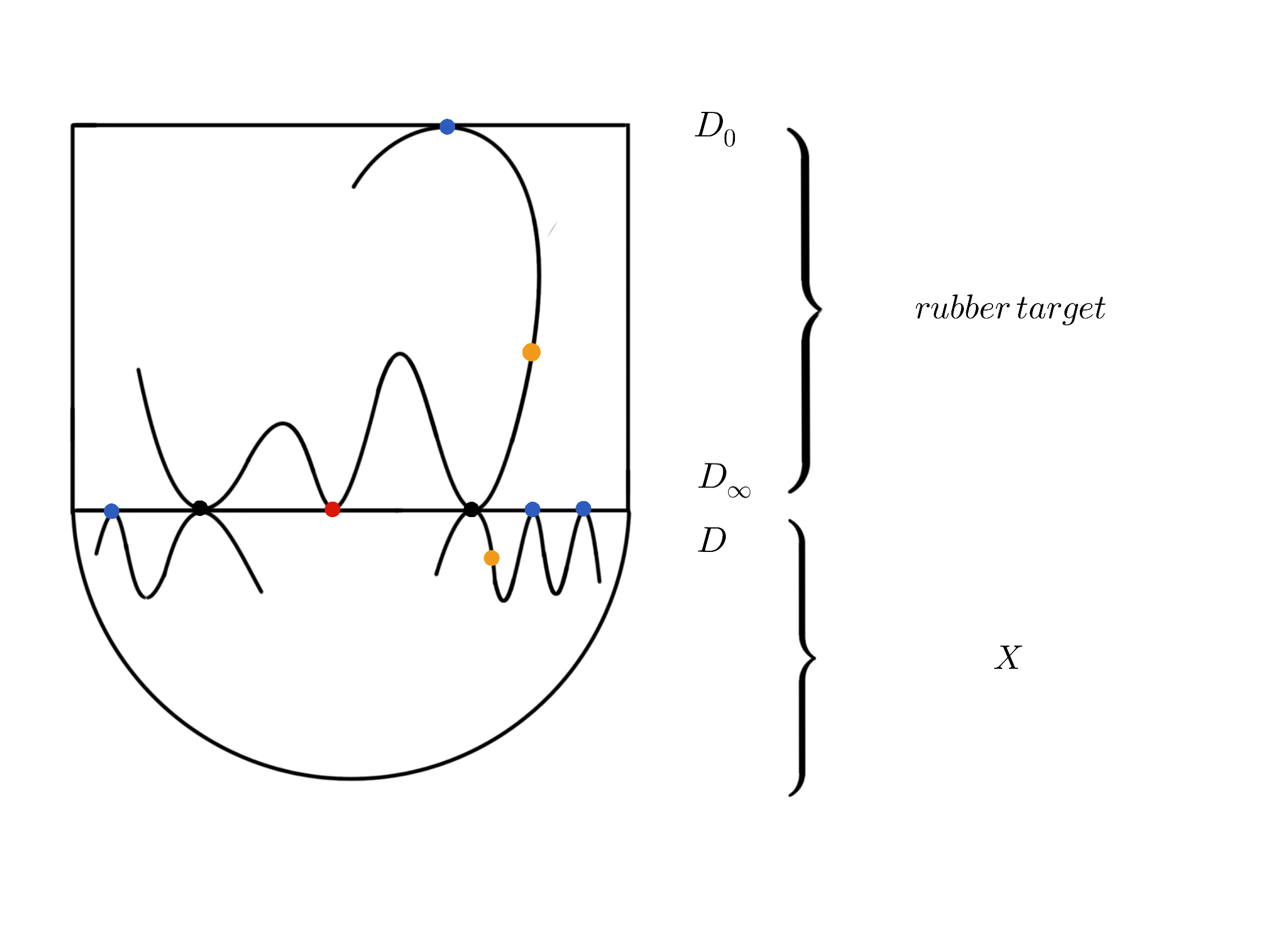}

The $0$-root becomes the relative marking at $D_0$ (in blue); the roots of marking type of $\Gamma^\infty$ become relative markings on components in $X$ falling on the ``crease" $D_\infty$ (in blue); the $\infty$-root of marking type of $\Gamma^0$ becomes the relative marking on the rubber component at $D_\infty$ (in red).

\end{ex}

Next, we need to introduce a few classes in Chow cohomology on certain moduli spaces. First, we consider $\bM^\bullet_{\Gamma^\infty}(X,D)$. According to \cite{GV}*{Section 2.5}, $\bM^{\bullet}_{\Gamma^\infty}(X,D)$ admits a map to $\mathcal T$(or \cite{Jun1}*{Section 4} as $\mathfrak Z^{\rel}$), the Artin stack of expanded degenerations of $(X,D)$. Also, by \cite{GV}*{Section 2.5}, $\mathcal T$ can be viewed as an open substack of $\mathfrak M_{0,3}$ (the Artin stack of prestable $3$-pointed rational curve) consisting of chains of curves that separate $\infty$ from $0,1$ (in notation of \cite{GV}). There is a divisor corresponding to cotangent lines at $\infty$, and we denote by $\Psi$ its pullback to $\bM^{
\bullet}_{\Gamma^{\infty}}(X,D)$. 

There is a divisor on $\bM^{\bullet}_{\Gamma^{\infty}}(X,D)$ corresponding to the locus where the target degenerates at least once (denoted by $\delta$).  
\begin{lem}\label{lem:psi}
$\Psi$ is linearly equivalent to $\delta$. 
\end{lem}
This can be proven by a standard comparison between $\Psi$ and the pullback of the corresponding psi-class on $\bM_{0,3}$ along the stabilization map $\mathcal T\rightarrow\bM_{0,3}$.
This allows for an explicit computation of $\Psi$, since $\delta$ can be constructed by a fiber product of relative moduli spaces and rubber moduli spaces (with multiplicities counted similarly as in the degeneration formula).

Let $t$ be a formal parameter. Given a $\Gamma^\infty$, we define
\begin{align}\label{neg-rel-infty}
C_{\Gamma^\infty}(t)=\dfrac{t}{t+\Psi}\in A^*(\bM^\bullet_{\Gamma^\infty}(X,D))[t^{-1}].
\end{align}

Next, we shift our focus to $\bM^\sim_{\Gamma_i^0}(D)$. Define
\[
c(l)=\Psi_\infty^l-\Psi_\infty^{l-1}\sigma_1+\ldots+(-1)^l\sigma_l,
\]
where $\Psi_\infty$ is the divisor corresponding to the cotangent line bundle determined by the relative divisor on $\infty$ side (besides \cite{GV}, also see \cite{MP}*{Section 1.5.2}). We then define
\[
\sigma_k=\sum\limits_{\{e_1,\ldots,e_k\}\subset \HE_{m,n}(\Gamma_i^0)} \prod\limits_{j=1}^{k} (d_{e_j}\bar\psi_{e_j}-\ev_{e_j}^*D),
\]
where $d_{e_j}$ is the absolute value of the weight at the root $e_j$. We adopt the convention that $\sigma_k=0$ if $k>|\HE_{m,n}(\Gamma_i^0)|$.
Here $\bar\psi_{e_j}$ is the pullback of the corresponding $\psi$ class from $\bM_{0,n+\rho_0+\rho_\infty}(D)$. $D$ in the pullback of evaluation map is interpreted as divisor class restricted to hypersurface, as explained at the beginning of the section. 

For each $\Gamma_i^0$, define
\begin{align}\label{neg-rel-0}
C_{\Gamma_i^0}(t)= \dfrac{\sum_{l\geq 0}c(l) t^{\rho_\infty(i)-1-l}}{\prod\limits_{e\in \HE_{n}(\Gamma_i^0)} \big(\frac{t+\ev_e^*D}{d_e}-\bar\psi_e\big) } \in A^*(\bM^\sim_{\Gamma_i^0}(D))[t,t^{-1}],
\end{align}
where $\rho_\infty(i)$ is the number of $\infty$-roots (of both types) associated with $\Gamma_i^0$.

\subsection{Relative Gromov--Witten cycle with negative contact points}\label{sec:rel-cycle}
In this subsection, we construct a cycle in $\bM_{0,n+\rho}(X,\beta)\times_{X^{\rho}} D^{\rho}$. In the next subsection, relative invariants are defined by the integration against this cycle.

For each $\fG$, we write
\begin{equation}\label{eqn:cg}
C_{\fG}=\left[ p_{\Gamma^\infty}^*C_{\Gamma^\infty}(t)\prod\limits_{\Gamma_i^0\in \mS_0} p_{\Gamma_i^0}^*C_{\Gamma_i^0}(t) \right]_{t^{0}}, 
\end{equation}
where 
$[\cdot]_{t^{0}}$ means taking 
the constant term, and $p_{\Gamma^\infty}, p_{\Gamma_i^0}$ are projections from $\prod\limits_{\Gamma^0_i\in \mS_0}\bM^\sim_{\Gamma^0_i}(D) \times \bM^{\bullet}_{\Gamma^\infty}(X,D)$ to corresponding factors. Recall
\[
\iota: \bM_\fG\rightarrow \prod\limits_{\Gamma^0_i\in \mS_0}\bM^\sim_{\Gamma^0_i}(D) \times \bM^{\bullet}_{\Gamma^\infty}(X,D)
\]
is the closed immersion from diagram \eqref{eqn:diag}.

\begin{defn}\label{def-rel-cycle}
Define \emph{the relative Gromov--Witten cycle of the pair $(X,D)$ of topological type $\Gamma$} to be 
\[\mathfrak c_\Gamma(X/D) = \sum\limits_{\fG \in \mathcal B_\Gamma} \dfrac{1}{|\Aut(\fG)|}(\mathfrak t_{\fG})_* ({\iota}^* C_{\fG} \cap [\bM_{\fG}]^{\vir}) \in A_*(\bM_{0,n+\rho}(X,\beta)\times_{X^{\rho}} D^{\rho}),\]
where $\iota$ is the vertical arrow in diagram \eqref{eqn:diag}.
\end{defn}

We would like to point out that $\infty$-roots of marking type in each $\fG$ should be regarded as markings ``of negative contact orders". Also note that due to the convention of $I$ (see Definition \ref{defn:locgraph}), the first $n$ markings correspond to interior markings, and the following $\rho$ markings correspond to relative markings (possibly of negative contact orders).

In fact, $\mathfrak c_\Gamma(X/D)$ is of pure dimension. If one traces back to the definitions of involved classes, it is not hard to compute the dimension of the relative Gromov--Witten cycle. One can check that the cycle $\mathfrak c_\Gamma(X/D)$ is of dimension
\[
d=\mathrm{dim}(X)-3+\int_{\beta} c_1(T_X(-\mathrm{log} D)) + n + \rho_+.
\]
Thus, the dimension matches with Proposition \ref{prop:vdim}, which can also be deduced from Theorem \ref{thm:compare}.

\begin{ex}[Relative Gromov--Witten cycle without negative markings]
When there are no negative markings, any $\Gamma_i^0$ does not have $\infty$-roots of marking type. Thus, all the $\infty$-roots are of node type, that is,
\[\rho_\infty=|\HE_n(\Gamma_i^0)|.\] 
It is straightforward to check that
\[C_{\Gamma_i^0}(t)=\bigg( \prod\limits_{e\in \HE_{n}(\Gamma_i^0)}d_e\bigg) t^{-1}+O(t^{-2}).\] 
For degree reasons, if $\fG$ has any graph of type $0$, $\mathfrak c_{\fG}$ would have $0$ as its constant term. As a result, $\mathfrak c_{\Gamma}(X/D)$ is simply $(\mathfrak t_{\fG})_*([\bM_\Gamma(X,D)]^{\vir})$ where $\fG$ is the bipartite graph without vertices on $0$-side. If we integrate insertions against this cycle, we recover the original definition of relative Gromov--Witten invariants.
\end{ex}
\begin{ex}[Relative Gromov--Witten cycle only $1$ negative marking]\label{ex:1neg}
For a given bipartite graph $\fG$, we denote the only $\infty$-root of marking type by $p$. Suppose that $p$ lies in the vertex $\Gamma^0$. Then $\rho_\infty=|\HE_n(\Gamma^0)|+1$. Note that 
\[C_{\Gamma^0}(t)=\prod\limits_{e\in \HE_{n}(\Gamma^0)}d_e+O(t^{-1}),\]
because $c_0=1$. If there exists another graph $\Gamma_i^0$ of type $0$, it would contribute 
\[t^{-1}\bigg(\prod\limits_{e\in \HE_{n}(\Gamma_i^0)}d_e+O(t^{-1})\bigg).\] 
So in order to get nonzero constant terms, we cannot allow more graphs of type $0$. As a result, $\mS_0=\{\Gamma^0\}$. Since $\Gamma^\infty$ may consist of multiple vertices, such bipartite graphs $\fG$ should look like the following.

\includegraphicsdpi{600}{}{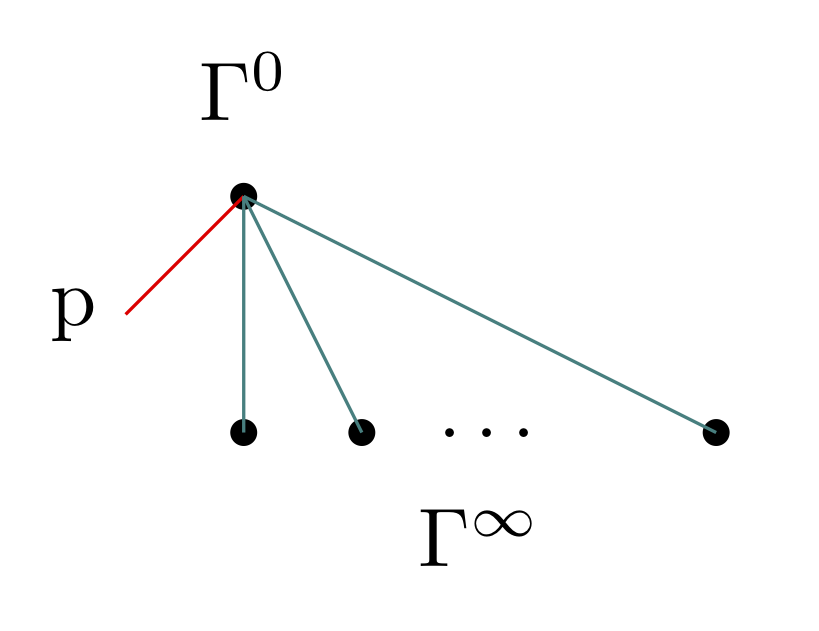}

Note that we suppress $0$-roots, legs, curve classes, etc. in order to make the picture cleaner. To sum it up, if one varies the number of vertices in $\Gamma^\infty$ and distribute decorations between $\Gamma^0, \Gamma^\infty$, we obtain all bipartite graphs $\fG$ that might contribute nontrivially to the relative Gromov--Witten cycle. Denote such set of bipartite graphs by $\mathcal B_\Gamma'$. For any $\fG\in \mathcal B_\Gamma'$, $C_{\fG}=\prod_{e\in \HE_{n}(\Gamma^0)}d_e$.
As a result, the relative Gromov--Witten cycle of topological type $\Gamma$ is simply
\[
\mathfrak c_\Gamma(X/D) = \sum\limits_{\fG \in \mathcal B_\Gamma'} \dfrac{\prod_{e\in \HE_{n}(\Gamma^0)}d_e}{|\Aut(\fG)|}(\mathfrak t_{\fG})_* \big([\bM_{\fG}]^{\vir}\big).
\]

\end{ex}
\begin{ex}[Relative Gromov--Witten cycle with two negative markings]\label{ex:2neg}
In this case, there are only three kinds of bipartite graphs which give nontrivial contributions to $\mathfrak c_\Gamma(X/D)$:

\includegraphicsdpi{600}{}{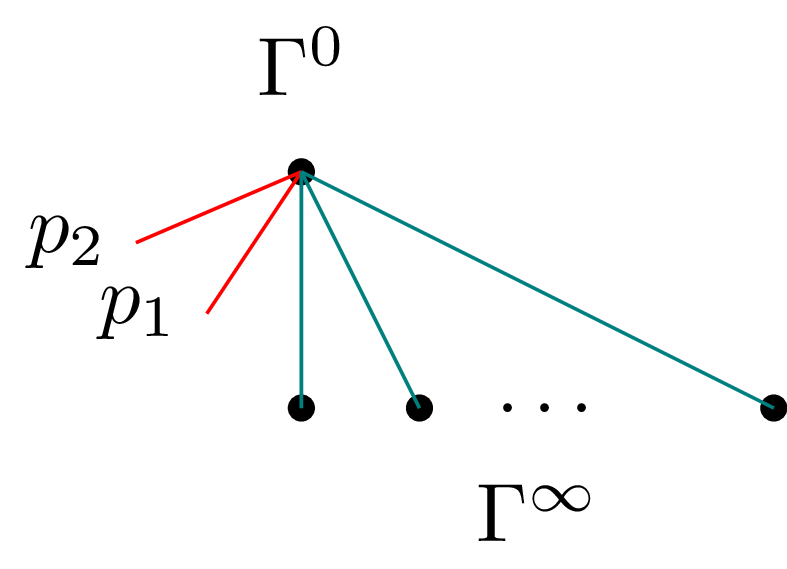}

\includegraphicsdpi{600}{}{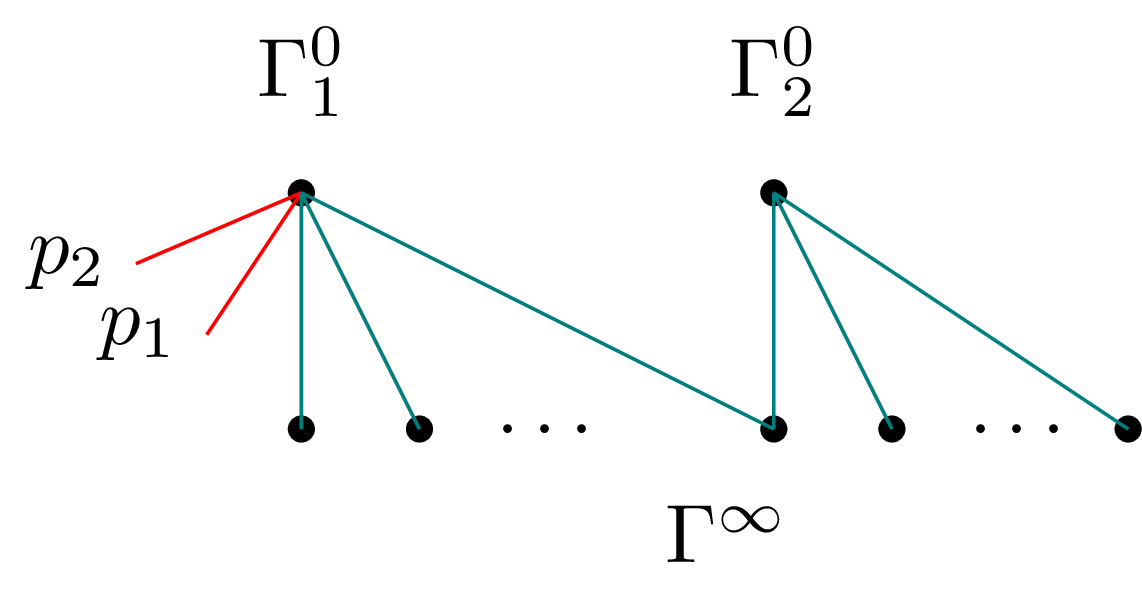}

\includegraphicsdpi{600}{}{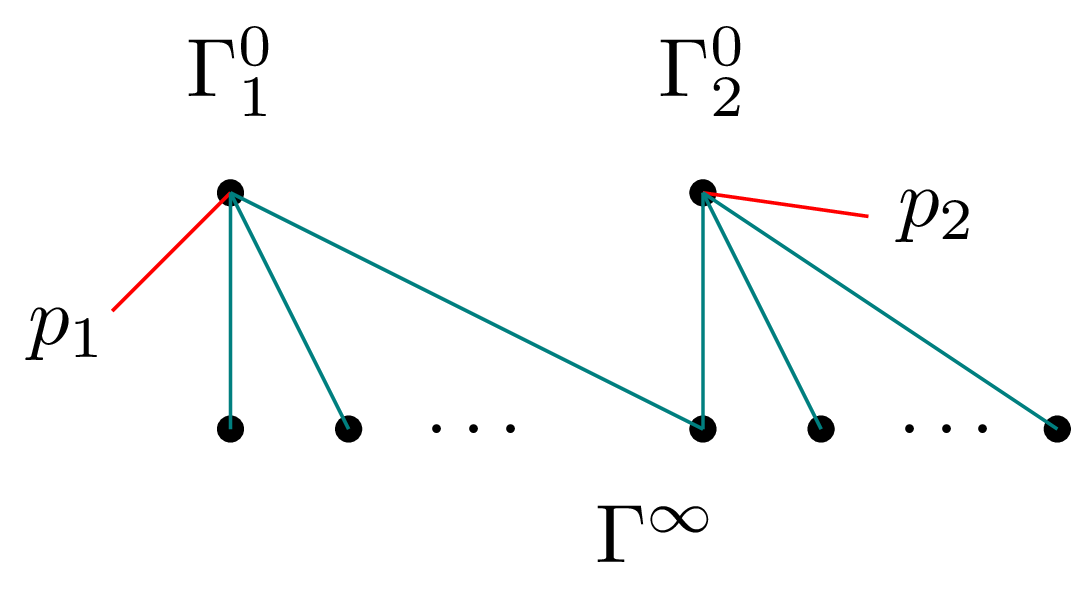}

Similar to Example \ref{ex:1neg}, we suppress $0$-roots, legs, curve classes, etc. By varying the number of vertices on $\Gamma^\infty$ and distributing decorations as before, we obtain all bipartite graphs that might contribute nontrivially.
Again as in Example \ref{ex:1neg}, the sets of bipartite graphs corresponding to the first, second and last pictures will be denoted by $\mathcal{B}_\Gamma^1$, $\mathcal{B}_\Gamma^2$ and $\mathcal{B}_\Gamma^3$, respectively. 
For the first picture, it is easy to compute that the total contribution $\mathfrak c_1$ is given by
\[\sum\limits_{\fG \in \mathcal B_\Gamma^1}\dfrac{\prod_{e\in \HE_{n}(\Gamma^0_1)}d_e}{|\Aut(\fG)|}(\mathfrak t_{\fG})_* \bigg(\Big(c(1)+\sum_{e\in \HE_{n}(\Gamma^0)}(d_e\bar\psi_e-ev_e^*D)-\Psi\Big)\cap[\bM_{\fG}]^{\vir}\bigg).\]
As for the second and last pictures, their total contributions are
\[\mathfrak c_i=\sum\limits_{\fG \in \mathcal B_\Gamma^i}\dfrac{\prod_{e\in \HE_{n}(\Gamma^0_1)}d_e\prod_{e\in \HE_{n}(\Gamma^0_2)}d_e}{|\Aut(\fG)|}(\mathfrak t_{\fG})_* \big([\bM_{\fG}]^{\vir}\big),\quad i=2,3.\]
So we have $\mathfrak c_\Gamma(X/D)=\sum_{i=1}^3 \mathfrak c_i.$
\end{ex}

\subsection{Relative invariants with negative contact orders}\label{rel-inv2}
Similar to Section \ref{rel-inv1}, we can define relative invariants by integrations against this cycle. For completeness, we state this straightforward definition of relative invariants, though most of the content is similar or parallel to Section \ref{rel-inv1}.

Let 
\begin{align*}
    \begin{split}
        \ualpha = (\bar\psi^{a_1}\alpha_1,\ldots,\bar\psi^{a_n}\alpha_n) 
        &\in (\C[\bar\psi] \otimes H^*(X))^{n},\\ \uepsilon = (\bar\psi^{b_1}\epsilon_1,\ldots,\bar\psi^{b_{\rho}}\epsilon_\rho) &\in (\C[\bar\psi] \otimes H^*(D))^{\rho}.
    \end{split}
\end{align*}
There are evaluation maps from $\bM_\fG$ corresponding to interior markings and relative markings
\begin{align*}
\ev_X=(\ev_{X,1},\ldots,\ev_{X,n}):\bM_\fG&\rightarrow X^n, \\
\ev_D=(\ev_{D,1},\ldots,\ev_{D,\rho}):\bM_\fG&\rightarrow D^\rho.
\end{align*}
Recall we have the stabilization map $t_\fG:\bM_\fG\rightarrow \bM_{0,n+\rho}(X,\beta)\times_{X^{\rho}} D^{\rho}$. There are also evaluation maps
\begin{align*}
\overline{\ev}_X=(\overline\ev_{X,1},\ldots,\overline\ev_{X,n}):\bM_{0,n+\rho}(X,\beta)\times_{X^{\rho}} D^{\rho}&\rightarrow X^n, \\
\overline{\ev}_D=(\overline\ev_{D,1},\ldots,\overline\ev_{D,\rho}):\bM_{0,n+\rho}(X,\beta)\times_{X^{\rho}} D^{\rho}&\rightarrow D^\rho,
\end{align*}
such that 
\[
\overline{\ev}_X\circ t_\fG=\ev_X, \quad \overline{\ev}_D\circ t_\fG=\ev_D.
\]

\begin{defn}\label{rel-inv-neg}
The {\it relative Gromov--Witten invariant of topological type $\Gamma$ with insertions $\uepsilon,\ualpha$} is
\[
\langle \uepsilon \mid \ualpha \rangle_{\Gamma}^{(X,D)} =  \displaystyle\int_{\mathfrak c_\Gamma(X/D)} \prod\limits_{j=1}^{\rho} \bar{\psi}_{D,j}^{b_j}\overline{\ev}_{D,j}^*\epsilon_j\prod\limits_{i=1}^n \bar{\psi}_{X,i}^{a_i}\overline{\ev}_{X,i}^*\alpha_i,
\]
where $\bar\psi_{D,j}, \bar\psi_{X,i}$ are pullback of psi-classes from $\bM_{0,n+\rho}(X,\beta)$ to $\bM_{0,n+\rho}(X,\beta)\times_{X^{\rho}} D^{\rho}$ corresponding to markings evaluated under $\overline\ev_D,\overline\ev_X$.
\end{defn}

\section{Matching the orbifold definition with the graph sum}\label{sec:loc}

In this section, we show that the two definitions coincide. As a corollary, relative Gromov--Witten invariants satisfy properties including string, divisor, dilaton, TRR and WDVV equations.

We will prove Theorem \ref{thm:limitexist} (Gromov--Witten cycles are independent of $r$ in orbifold definition) as well as the following comparison between the two definitions.
\begin{thm}\label{thm:compare}
The two definitions of Gromov--Witten cycles agree. More precisely, fix a topological type $\Gamma=(0,n,\beta,\rho,\vec{\mu})$. For $r\gg 1$, we have the following equality between cycle classes
\[
\mathop{\mathrm{lim}}_{r\rightarrow \infty} r^{\rho_-}\tau_*([\bM_\Gamma(X_{D,r})]^{\vir}) = \sum\limits_{\fG \in \mathcal B_\Gamma} \dfrac{1}{|\Aut(\fG)|}(\mathfrak t_{\fG})_* ({\iota}^* C_{\fG} \cap [\bM_{\fG}]^{\vir})
\]
\end{thm}
A very brief outline of the proof is as follows. We first use degeneration formula to reduce the theorem to the case when $X$ is a $\Pp^1$-bundle and $D$ is an invariant section. Next, we apply localization formula to the $\Pp^1$-bundle under fiberwise action, and express it as a graph sum. We shall see that the graph sum is independent of $r$ for large $r$, thus proving Theorem \ref{thm:limitexist}. By a more precise analysis, we can also identify the localization formula with the right-hand side of the equation in Theorem \ref{thm:compare}.

\subsection{Degeneration formula for the virtual cycle}\label{sec:deg}

First, we need to introduce some notation. Let $L$ be a line bundle over $D$. Define
\[
P:=\Pp_D(\sO\oplus L).
\]
Let $P_{D_0,r}$ be the root stack of $P$ by applying $r$-th root construction along the zero section $D_0$. The coarse moduli of $P_{D_0,r}$ is still $P$.
There is a substack in $P_{D_0,r}$ isomorphic to $\sqrt[r]{D/L}$ that lies over $D_0$. We denote this substack by $\cD_0$. On the other hand, root construction does not modify $D_\infty\in P$. We denote its preimage in $P_{D_0,r}$ by $\cD_\infty$. The $\GM$-action lifts to $P_{D_0,r}$, and $\cD_0, \cD_\infty$ are invariant loci.
In this subsection, we consider the pair $(X,D)$, and focus on the case when $L=N_{D/X}$ is the normal bundle of $D\subset X$. 

We consider the degeneration of $X_{D,r}$ to the normal cone of $\cD_{0}$:
\begin{equation}\label{deg-root-stack}
X_{D,r}\leadsto P_{D_0,r} \cup_D X.
\end{equation}
Here the union glues together the infinity divisor $\cD_\infty\subset P_{D_0,r}$ and the divisor $D\subset X$. More explicitly, the degeneration (\ref{deg-root-stack}) can be obtained by first degenerating $X$ to the normal cone of $D$ (as a blow-up of $X\times \A^1$), and then taking the $r$-th root stack along the strict transform of $D\times \A^1$.

In \cite{Jun1}*{Definition 4.11}, admissible triples $(\Gamma_1,\Gamma_2,I)$ are introduced ($\Gamma_1,\Gamma_2$ are admissible graphs with matching roots). Suppose that there are $k_1, k_2$ legs in $\Gamma_1, \Gamma_2$, respectively. $I$ is an order preserving inclusion $I:[k_1]\rightarrow [k_1+k_2]$. In order to state the degeneration formula, $\Gamma_1,\Gamma_2$ can be glued along the roots, and $I$ is used to reorder the legs from $\Gamma_1,\Gamma_2$. We use $\operatorname{Aut}(\mathfrak i)$ to denote the automorphism group of the gluing of $\Gamma_1,\Gamma_2$. We do not spell out the formal language of defining admissible graphs and the gluing, because they will not be used in the rest of the paper. But we refer interested readers to \cites{Jun1, Jun2} directly. In \cites{Jun1,Jun2}, $\eta$ is used to denote admissible triples. But we would like to reserve $\eta$ for partitions, and use $\mathfrak i=(\Gamma_1,\Gamma_2,I)$ instead.

For an admissible triple $\mathfrak i=(\Gamma_1,\Gamma_2,I)$, let 
\[
\vec{\eta}=(\eta_1,\ldots,\eta_{\rho_\eta})\in (\mathbb Z_{>0})^{\rho_\eta}
\]
be a partition defined by weights of roots in $\Gamma_1$ (or equivalently $\Gamma_2$, since they have matching roots), where $\rho_\eta$ is the length of the partition (that is, number of roots).
We define the morphisms $\tau_{\mathfrak i}$ and $\Delta$ as follows. 
First, define $\bM_{\Gamma,\mathfrak i}$ by the Cartesian diagram
\begin{equation}
\xymatrix{
\bM_{\Gamma,\mathfrak i} \ar[r]^{} \ar[d] & \bM_{\Gamma_1}^{\bullet}(P_{D_0,r},\cD_\infty)\times \bM_{\Gamma_2}^{\bullet}(X,D) \ar[d] \\
D^{\rho_\eta} \ar[r]^{\Delta} & (D\times D)^{\rho_\eta}.
}                                 
\end{equation}
Let
\[
\tau_\mathfrak i:\bM_{\Gamma,\mathfrak i}\rightarrow \bM_{0,n+\rho}(X,\beta)\times_{X^{\rho}} D^{\rho}
\]
be the forgetful map.
The degeneration formula is proven by using the moduli stack of stable morphisms to the stack of expanded degenerations (under the notation $\mathfrak M(\mathfrak W,\Gamma)$ in \cite{Jun2}). The total space of \eqref{deg-root-stack} admits a morphism to $X$ (mapping to its coarse moduli, contracting the exceptional divisor and then projecting). This induces a morphism from the moduli stack to $\bM_{0,n+\rho}(X,\beta)\times_{X^{\rho}} D^{\rho}$ since markings with orbifold structures must be sent to $D$. By comparing the general fibers and the fiber at $0$, one obtains the following version of the degeneration formula in the Chow group $A_*(\bM_{0,n+\rho}(X,\beta)\times_{X^{\rho}} D^{\rho})$:
\begin{align}\label{equ-deg}
\begin{split}
    &\tau_*[\bM_\Gamma(X_{D,r})]^{\vir} \\
    \sim&\sum \dfrac{\prod (\eta_i)}{|\operatorname{Aut}(\mathfrak i)|} \tau_{\mathfrak i*}\Delta^!([\bM_{\Gamma_1}^{\bullet}(P_{D_0,r},\cD_\infty)]^{\vir}\times [\bM_{\Gamma_2}^{\bullet}(X,D)]^{\vir}),
\end{split}
\end{align}
where $\sim$ means rational equivalence, and the sum ranges over all admissible triples $\mathfrak i=(\Gamma_1,\Gamma_2,I)$ .
Here $|\operatorname{Aut}(\mathfrak i)|$ is the order of the automorphism group $\operatorname{Aut}(\mathfrak i)$.

Theorem \ref{thm:compare} then reduces to the localization analysis on equivariant virtual cycles $[\bM_{\Gamma_1}^{\bullet}(P_{D_0,r},\cD_\infty)]^{\vir}_{\mathbb{C}^*}$ in the next section. 
The pair $(P_{D_0,r},\cD_\infty)$ is called \emph{the relative local model} for $X_{D,r}$ in \cite{TY}. 

\begin{rmk}
It is worth mentioning that our presentation is slightly different from the presentation in \cite{TY}. However, two presentations are essentially equivalent. The authors of \cite{TY} were emphasizing on the idea of reducing the comparison between relative and orbifold invariants to relative local models. Therefore, they also wrote down the degeneration formula for relative invariants and compared it with the degeneration formula for orbifold invariants. Then, the comparison between relative and orbifold invariants reduced to the computation of invariants of the relative local model $(P_{D_0,r},\cD_\infty)$ as well.
\end{rmk}

\subsection{Localization of the relative local model}\label{sec:localmodel}
In the case of relative local models, we have to consider stable maps with orbifold structures on $\cD_0$ as well as tangency conditions on $\cD_\infty$. A topological type of such stable maps should contain more information. First, we must have $g,n,\beta$ as before. Recall that $n$ is the number of interior markings without orbifold structures. Denote
\[
\displaystyle d_0=\int_\beta D_0, \quad d_\infty=\int_\beta D_\infty .
\]
Let $\vec{\mu}=(\mu_1,\ldots,\mu_\rho)$ be a tuple of nonzero integers such that
\begin{align*}
\sum_{i=1}^\rho\mu_i=d_0.
\end{align*}
As in Convention \ref{conv:Gamma}, this notation indicates that we have $\rho$ markings with orbifold structures (thus mapping to $\cD_0$) and the partition $\vec{\mu}$ describes the age information of these markings. Moreover, let $\vec{\eta}=(\eta_1,\ldots,\eta_{\rho_\eta})$ be a tuple of positive integers such that
\[
\sum_{i=1}^{\rho_\eta}\eta_i=d_\infty.
\]
This indicates that we have $\rho_\eta$ relative markings with contact orders given by $\vec{\eta}$. Putting all these ingredients together, the topological type is denoted by $\hat\Gamma=(0,n,\beta,\rho,\vec{\mu}, \rho_\eta,\vec{\eta})$.
We denote the moduli space of relative orbifold stable maps to $(P_{D_0,r},\cD_\infty)$ of topological type $\hat\Gamma$ by $\bM_{\hat\Gamma}^{\bullet}(P_{D_0,r},\cD_\infty)$. We refer to \cites{TY,TY18} for more details about this kind of hybrid stable maps.

There is a map $P_{D_0,r}\rightarrow P$ to the coarse moduli space, and a projection $P\rightarrow D$ to the base of the projective bundle. The composition of these two induces a stabilization map
\[
\tau:\bM_{\hat\Gamma}^{\bullet}(P_{D_0,r},\cD_\infty)\rightarrow \bM_{0,n+\rho+\rho_\eta}^{\bullet}(D).
\]
This map forgets relative and orbifold conditions of stable maps.

We compute the $\GM$-equivariant class 
\begin{equation}\label{eqn:cls}
\tau_*[\bM_{\hat\Gamma}^{\bullet}(P_{D_0,r},\cD_\infty)]^{\vir}_{\GM}
\end{equation}
by the virtual localization formula where $r\gg 1$ (see also \cite[Section 3]{JPPZ18} for more details on the localization analysis).

Let $\bF_{\fG}$ be the union of components which parameterizes $\GM$-invariant relative orbifold stable maps of type $\fG$, where $\fG$ is a localization bipartite graph in Definition \ref{def:loc-graph}. Recall that we allow $\Gamma^\infty$ in $\fG$ to be unstable, that is, each of the vertices has $0$ as its curve class with one $0$-root, one $\infty$-root and nothing else (see Remark \ref{rmk:unstable}). We associate this bipartite graph with the corresponding orbifold stable maps where the target does not degenerate at $\infty$. A precise description of $\bF_{\fG}$ is as follows.
\begin{itemize}
\item If the target degenerates at $\infty$, $\bF_\fG$ admits an \'etale cover by
\[
\prod\limits_{\Gamma^0_i\in \mS_0}\bM_{\Gamma^0_i}(\cD_0) \times_{D^{|E|}} \bM^{\bullet\sim}_{\Gamma^\infty}(D)
\]
where $\times_{D^{|E|}}$ is the fiber product identifying evaluation maps according to edges (similar to \eqref{eqn:fiberprod}). If $\Gamma^0_i$ is unstable (see Definition \ref{def:loc-graph}), we simply set $\bM_{\Gamma^0_i}(\cD_0)=D$. In this case, the virtual cycle  $[\bF_\fG]^{\vir}$ becomes
\[\frac{1}{\prod_{e\in \HE_{0}(\Gamma^\infty)}d_e}\Delta^!\left(\prod\limits_{\Gamma^0_i\in \mS_0}[\bM_{\Gamma^0_i}(\cD_0)]^{\vir}\times [\bM^{\bullet\sim}_{\Gamma^\infty}(D)]^{\vir}\right).\]
Since roots of node type form edges, those $d_e$ are given by the absolute values of their weights. You may think these $d_e$ as ``multiplicities of the edges".

\item If the target does not degenerate at $\infty$, $\bF_{\fG}$ admits an \'etale cover by
\[
 \prod\limits_{\Gamma^0_i\in \mS_0}\bM_{\Gamma^0_i}(\cD_0).
\]
In this case,  
\[[\bF_\fG]^{\vir}=\frac{1}{\prod_e d_e}\prod\limits_{\Gamma^0_i\in \mS_0}[\bM_{\Gamma^0_i}(\cD_0)]^{\vir}.\]
\end{itemize}

The localization residue of the graph $\fG$ can be described schematically as follows:
\[
\operatorname{Res}_{\fG}=\left(\prod\limits_{\Gamma^0_i\in \mS_0} \operatorname{Cont}(\Gamma^0_i)\right)\operatorname{Cont}(\Gamma^\infty).
\]
Note that we do not have edge contributions since the degree $\frac{d_e}{r}$ of $f^*T(-\cD_\infty)$ ($T$ denotes the relative tangent bundle of the projection $P_{D_0,r}\rightarrow D$) is less than 1 for $[f]\in \bF_{\fG}$ (see \cite{JPT}*{Section 2.2} or \cite{JPPZ18}*{Section 3.4}).

The virtual localization formula is 
\begin{align}\label{eqn:locform}
\begin{split}
    &\tau_*[\bM_{\hat\Gamma}(P_{D_0,r},\cD_\infty)]^{\vir}_{\GM}\\
    =&\sum_\fG \frac{1}{|\operatorname{Aut}(\fG)|}\tau_*(\operatorname{Res}_\fG\cap[\bF_\fG]^{\vir}).
\end{split}
\end{align}
 
On the right-hand side of \eqref{eqn:locform}, we slightly abuse the notation by using $\tau$ to mean the restriction of $\tau$ on the corresponding $\bF_\fG$. If the target does not degenerate at $\infty$, the contribution at $\infty$ is trivial:
\[
\operatorname{Cont}(\Gamma^\infty)=1.
\]
If the target degenerates at $\infty$, we have
\begin{align*}
\operatorname{Cont}(\Gamma^\infty)=& \frac{\prod_{e\in \HE_{0}(\Gamma^\infty)}d_e}{-t-\Psi_0},
\end{align*}
where $\Psi_0$ is the divisor corresponding to the cotangent line bundle determined by the relative divisor on $0$ side. Comparing with $[\bF_\fG]^{\vir}$ in \eqref{eqn:locform}, we see that the factor $\prod_{e\in \HE_0(\Gamma^\infty)}d_e$ is cancelled.

Let $v_i$ be the single vertex of $\Gamma^0_i$ which is stable. We write $E_i$ for the set of edges associated with the vertex $v_i$ and write $\rho_{-}(i)$ for the number of $\infty$-roots of marking type (corresponding to large age markings) associated with $v_i$. The localization contribution is
\begin{align}\label{loc-con-0}
\begin{split}
& \operatorname{Cont}(\Gamma^0_i)\\
=&\prod_{e\in E_i}\frac{rd_e}{(t+\ev_e^*c_1(L))-d_e\bar{\psi}_e}\left(\sum_{j\geq 0}(t/r)^{|E_i|+\rho_{-}(i)-1-j}c_{j}(-R^*\pi_*\mathcal L_r)\right)\\
= &\frac{\prod_{e\in E_i}d_e}{\prod_{e\in E_i}(t+ \ev_e^*c_1(L)-d_e\bar{\psi}_e)}\left(\sum_{j\geq 0}\frac{t^{|E_i|+\rho_{-}(i)-1-j}}{r^{\rho_{-}(i)-1-j}}c_{j}(-R^*\pi_*\mathcal L_r)\right),
\end{split}
\end{align}
where 
\begin{itemize}
    \item the factor
\[
\frac{rd_e}{(t+\ev_e^*c_1(L))-d_e\bar{\psi}_e}=\frac{1}{\frac{(t+\ev_e^*c_1(L))}{rd_e}-\frac{\bar{\psi}_e}{r}}
\]
comes from the inverse normal bundle of smoothing the node connecting the edge $e$ and the vertex $v_i\in\Gamma^0_i$,
\item and the morphism
\[
\pi: \mathcal C_{\Gamma^0_i}(\cD_0)\rightarrow \bM_{\Gamma^0_i}(\cD_0)
\]
 is the projection from the universal curve. There is a universal $r$-th root bundle $L_r$ over the gerbe $\cD_0$ and the evaluation map 
\[
f:\mathcal C_{\Gamma^0_i}(\cD_0) \rightarrow \cD_0.
\] Our $\mathcal L_r$ is defined as 
\[\mathcal L_r=f^*L_r.\]
\end{itemize}

The virtual rank of $-R^*\pi_*\mathcal L_r\in K^0(\bM_{\Gamma^0_i}(\cD_0))$ is $|E_i|+\rho_{-}(i)-1$. Note that when matching with relative invariants, $|E_i|+\rho_{-}(i)$ corresponds to $\rho_\infty(i)$ in (\ref{neg-rel-0}). The vertex contribution is
\[
\sum_{j\geq 0}(t/r)^{|E_i|+\rho_{-}(i)-1-j}c_{j}(-R^*\pi_*\mathcal L_r).
\]

By Lemma \ref{lemma:HHI}, we have
\begin{align*}
&r^{\rho_-(i)}\tau_*\left(\operatorname{Cont}(\Gamma^0_i)\cap [\bM_{\Gamma_i^0}(\cD_0,\beta)]^{\vir}\right)\\
  \notag  =&\tau_*\Bigg(\frac{\left(\prod_{e\in E_i}d_e\right)}{\prod_{e\in E_i}(t+ \ev_e^*c_1(L)-d_e\bar{\psi}_e)}\cdot\\
  \notag & \sum_{j\geq 0}\frac{t^{|E_i|+\rho_-(i)-1-j}}{r^{-1-j}}c_{j}(-R^*\pi_*\mathcal L_r)\cap [\bM_{\Gamma_i^0}(\cD_0,\beta)]^{\vir}\Bigg)\\
  \notag = &\frac{\left(\prod_{e\in E_i}d_e\right)}{\prod_{e\in E_i}(t+ \ev_e^*c_1(L)-d_e\psi_e)}\sum_{j\geq 0}t^{|E_i|+\rho_-(i)-1-j}\\
  \notag  &(\tau_2)_*\left(\left( \Psi_\infty^{j}-\Psi_\infty^{j-1}\sigma_1+\ldots+(-1)^{j}\sigma_{j}\right)\cap [\bM^\sim_{\Gamma_i^0}(D)]^{\vir}\right),
\end{align*}
for $r$ sufficiently large, where $\tau_2$ is the forgetful map from $\bM^\sim_{\Gamma_i^0}(D)$ to the corresponding moduli space of stable maps to $D$ (see the Appendix). Hence, 
\[r^{\rho_-(i)}\tau_*\left(\operatorname{Cont}(\Gamma^0_i)\cap [\bM_{\Gamma_i^0}(\cD_0,\beta)]^{\vir}\right)\]
is simply 
\[(\tau_2)_*\left(C_{\Gamma^0_i}(t)\cap [\bM^\sim_{\Gamma_i^0}(D)]^{\vir}\right).\]
Recall that $C_{\Gamma^0_i}(t)$ is defined in Section \ref{sec:rel-neg-set-up} as follows:
\begin{align*}
C_{\Gamma_i^0}(t)= \dfrac{\sum_{l\geq 0}c(l) t^{\rho_\infty(i)-1-l}}{\prod\limits_{e\in \HE_{n}(\Gamma_i^0)} \big(\frac{t+\ev_e^*D}{d_e}-\bar\psi_e\big) } \in A^*(\bM^\sim_{\Gamma_i^0}(D))[t,t^{-1}],
\end{align*}
where 
\[
c(l)=\Psi_\infty^l-\Psi_\infty^{l-1}\sigma_1+\ldots+(-1)^l\sigma_l.
\]

Now if the single vertex $v_i$ of $\Gamma^0_i$ is unstable, then
by Lemma 12 in \cite{JPPZ18} we know that such unstable vertex could occur only when there is only one $0$-root and one $\infty$-root of node type. So $\rho_-(i)=0$. And in this case, $\operatorname{Cont}(\Gamma^0_i)=1$. 

The coefficient of $t^0$ from the contribution of all vertices is
\begin{align*}
    &\left[\left(\prod_{\Gamma^0_i}r^{\rho_-(i)}\operatorname{Cont}(\Gamma^0_i)\right)\operatorname{Cont}(\Gamma^\infty)\right]_{t^0},
\end{align*}
where $[]_{t^0}$ means taking the coefficient of $t^0$.
Therefore,
\begin{align}\label{equ:rel-local}
\begin{split}
    &r^{\rho_-}\tau_*\left([\bM_{\hat\Gamma}(P_{D_0,r},\cD_\infty)]^{\vir} \right)\\
 =&\sum_{\fG}\frac{1}{|\operatorname{Aut}(\fG)|}\tau_*\left( \left[\left(r^{\rho_-(i)}\prod\limits_{\Gamma^0_i\in \mS_0} \operatorname{Cont}(\Gamma^0_i)\right) \operatorname{Cont}(\Gamma^\infty)\right]_{t^0} \cap[\bF_\fG]^{\vir}\right).
\end{split}
\end{align}

Theorem \ref{thm:compare} follows from Definition \ref{def-rel-cycle} of relative cycle $\mathfrak c_\Gamma(X/D)$, the degeneration formula (\ref{equ-deg}) and the identity (\ref{equ:rel-local}). More precisely, the localization contribution at $0$ corresponds to (\ref{neg-rel-0}). The localization contribution at $\infty$ together with the virtual class $[\bM_{\Gamma_2}^{\bullet}(X,D)]^{\vir}$ from the degeneration formula (\ref{equ-deg}) corresponds to (\ref{neg-rel-infty}). Indeed, the class $C_{\Gamma^\infty}(t)$ in (\ref{neg-rel-infty}) should actually be read as follows in this context:
\begin{align}\label{C_gamma^infty}
C_{\Gamma^\infty}(t)=\dfrac{t}{t+\Psi}=1+\frac{\Psi}{-t-\Psi}\in A^*(\bM^\bullet_{\Gamma^\infty}(X,D))[t^{-1}]
\end{align}
where the $\Gamma^\infty$ in (\ref{neg-rel-infty}) should be seen as the gluing of $\Gamma^\infty$ in (\ref{equ:rel-local}) with $\Gamma_2$ along the corresponding roots. The term $1$ in (\ref{C_gamma^infty}) corresponds to the case when the target does not degenerate at $\infty$ in localization computation. The term $\frac{\Psi}{-t-\Psi}$ corresponds to the case when the target degenerates at $\infty$ in localization computation. By Lemma \ref{lem:psi}, the $\Psi$-class in the numerator of (\ref{C_gamma^infty}) is linear equivalent to a cycle on $\bM^\bullet_{\Gamma^\infty}(X,D)$ which is supported on relative stable maps whose targets degenerate at least once. In this way, the localization contribution at $\infty$ and the virtual class $[\bM_{\Gamma_2}^{\bullet}(X,D)]^{\vir}$ from the degeneration formula (\ref{equ-deg}) are encoded in $C_{\Gamma^\infty}(t)$.

We note that all those graphs of type $0$ in (\ref{neg-rel-0}) are stable while some of the graphs of type $0$ in (\ref{equ:rel-local}) are allowed to be unstable. To match the formulas, we need to contract those edges connected to those unstable vertices.

\section{Relative quantum rings and Givental formalism}\label{sec:formalism}
In this section, we will build a theory of quantum ring and an analog of Givental formalism. We will further describe a version of Virasoro operators as quantized operators. Then Virasoro operators annihilate the genus-zero part of total potential simply due to the axioms of the Lagrangian cones.

Similar to Section \ref{sec:rel-neg} (beginning of Section \ref{sec:rel-neg-set-up}), we abuse the notation by using $D$ for both its divisor class in cohomology $H^2(X)$, and its restriction in $H^2(D)$.

\subsection{The ring of insertions}\label{sec:ins} We first define a vector space containing all possible insertions of the relative theory. It will then be given a ring structure which later becomes the classical limit of the relative quantum product.

Define $\HH_0=H^*(X)$ and $\HH_i=H^*(D)$ if $i\in \Z - \{0\}$.
Let
\[
\HH=\bigoplus\limits_{i\in\Z}\HH_i.
\]
Each $\HH_i$ naturally embeds into $\HH$. For an element $\alpha\in \HH_i$, we denote its image in $\HH$ by $[\alpha]_i$. Define a pairing on $\HH$ by the following.
\begin{equation}\label{eqn:pairing}
\begin{split}
([\alpha]_i,[\beta]_j) = 
\begin{cases}
0, &\text{if } i+j\neq 0,\\
\int_X \alpha\cup\beta, &\text{if } i=j=0, \\
\int_D \alpha\cup\beta, &\text{if } i+j=0, i,j\neq 0.
\end{cases}
\end{split}
\end{equation}
The pairing on the rest of the classes is generated by linearity. We pick a basis $\{T_k\}$ for $H^*(X)$, and a basis $\{\bar T_k\}$ for $H^*(D)$. Using these, we define a basis for $\HH$. Let 
\begin{align*}
\widetilde T_{0,k}&=[T_k]_0, \\
\widetilde T_{i,k}&=[\bar T_k]_i \text{ when } i\neq 0.
\end{align*}
Let $\{T^k\}$ be the dual basis of $\{T_k\}$ under Poincar\'e pairing of $H^*(X)$, and $\{\bar T^k\}$ be the dual basis of $\{\bar T_k\}$ under Poincar\'e pairing of $H^*(D)$. Define
\begin{align*}
\widetilde T_{0}^k&=[T^k]_0, \\
\widetilde T_{i}^k&=[\bar T^k]_i \text{ when } i\neq 0.
\end{align*}
In these definitions, the set $\{\widetilde T_{i,k}\}$ forms a basis of $\HH$, and the set $\{\widetilde T_{-i}^k\}$ forms its dual basis via the pairing just defined. The ranges of $i,k$ is understood in the obvious way. We will not specify their ranges in later text unless there is an extra condition. We want to emphasize that in the pairing of $\HH$, the dual of $\widetilde T_{i,k}$ is $\widetilde T_{-i}^k$. Note the negative sign on $i$ via the dualization.

In order to define a ring structure, we first describe a tri-linear form $A$ generated by the following.
\begin{equation}\label{eqn:trilinear}
    A([\alpha]_{i},[\beta]_j,[\gamma]_l)=
    \begin{cases}
        0 &\text{if } i+j+l\neq 0, \\
        \int_X \alpha\cup\beta\cup\gamma &\text{if } i=j=l=0, \\
        \int_D \alpha\cup\beta\cup\gamma &\text{if } i+j+l=0 \text{ and only one of $i,j,l$ is negative}, \\
        \int_D D\cup\alpha\cup\beta\cup\gamma &\text{if } i+j+l=0 \text{ and two of $i,j,l$ are negative}.
    \end{cases}
\end{equation}
In the third case, if one of $i,j,k$ is $0$, we assume the corresponding class is restricted to $D$ before the integration. 
\begin{defn}
Define a ring structure on $\HH$ by
\begin{equation*}
    [\alpha]_i\cdot [\beta]_j = \sum\limits_{l,k} A([\alpha]_i,[\beta]_j,\widetilde T_{l,k})\widetilde T_{-l}^k.
\end{equation*}
\end{defn}

In terms of computation, readers can refer to the following description of the product structure. With a slight abuse of notation, we set $\iota:D\rightarrow X$ to be the inclusion. And we use $\iota_!$ to denote the Gysin pushforward.
By a straightforward computation, one can show that
\[
[\alpha]_i \cdot [\beta]_j =
\begin{cases}
[\iota^*\alpha\cup\beta]_{i+j} &\text{if } i=0,j\neq 0, \\
[\alpha\cup\iota^*\beta]_{i+j} &\text{if } i\neq 0, j=0, \\
[\iota_!(\alpha\cup\beta)]_{i+j} &\text{if } i,j\neq 0 \text{ and } i+j = 0, \\
[\alpha\cup \beta]_{i+j} &\text{if } i,j\neq 0, i+j<0 \text{ and one of $i,j$ is positive},\\
[D\cup\alpha\cup\beta]_{i+j} &\text{if } i,j\neq 0, i+j>0 \text{ and one of $i,j$ is negative,}\\
[\alpha\cup\beta]_{i+j} &\text{if } i,j>0, \\
[D\cup \alpha\cup \beta]_{i+j} &\text{if } i,j<0.
\end{cases}
\]
In the product, $\HH$ is in fact a bigraded ring. One obvious grading is to regard a class $[\alpha]_i$ as degree $i$. We choose the notation to be the following.
\begin{equation}\label{eqn:deg1}
\deg^{(1)}([\alpha]_i)=i.
\end{equation}
To define the other grading, suppose that $\alpha\in \HH_i$ is a cohomology class of (real) degree $d$ (that is, in $H^d(X)$ if $i=0$, or in $H^d(D)$ if $i\neq 0$). Define
\begin{equation}\label{eqn:deg2}
\deg^{(2)}([\alpha]_i)=
\begin{cases}
d/2 &\text{if } i\geq 0,\\
d/2+1 &\text{if } i<0.
\end{cases}
\end{equation}
One can easily check that the product preserves this grading as well.

\begin{rmk}\label{rmk:heuristic}
This product structure and the second grading has a complicated and unnatural look. 
But there is a natural way to think of this product if we heuristically ignore cohomology classes in $H^*(D)$ that are not restrictions of classes in $H^*(X)$.  For an integer $i>0$, we regard $\HH_{-i}$ as the image $\iota_! H^*(D)$ inside $H^*(X)$. By our heuristic assumption, an element of it can always be written as $\iota_!(\iota^*\alpha)=D\cup\alpha\in H^*(X)$ for some $\alpha\in H^*(X)$. 
 
Consider the case when an element $\iota_!(\iota^*\alpha)\in\HH_{<0}$ multiplies an element $\beta\in\HH_{\geq 0}$. If the answer lies in $\HH_{\geq 0}$, we have
\[
\iota_!(\iota^*\alpha) \cup \beta = D\cup\alpha\cup\beta \in \HH_{\geq 0}.
\]
If the answer lies in $\HH_{<0}$, we have
\[
\iota_!(\iota^*\alpha)\cup\beta=\iota_!(\iota^*\alpha\cup\beta) \in \iota_!\HH_{<0}.
\]
Thus, we no longer have the extra $D$.
When two classes $\iota_!(\iota^*\alpha), \iota_!(\iota^*\beta)\in \HH_{<0}$ multiply together, we have
\[
\iota_!(\iota^*\alpha)\cup \iota_!(\iota^*\beta)=D\cup D\cup\alpha\cup\beta=\iota_!(D\cup\iota^*\alpha\cup\iota^*\beta)\in \iota_!\HH_{<0}.
\]
This also explains the shifting of $\deg^{(2)}$ in $\HH_{<0}$ in terms of the Gysin pushforward.
\end{rmk}
\subsection{A different notation of relative invariants}
In order to match with the ring of insertions $\HH$, we need a different notation of relative Gromov--Witten invariants with descendants. To emphasize its difference with Definition \ref{rel-inv-neg}, we use a more traditional notation $I_\beta(\ldots)$ appeared in, for example, \cite{FP}.

Given a curve class $\beta$ and insertions $[\alpha_1]_{i_1}\bar\psi^{a_1}, \ldots, [\alpha_n]_{i_n}\bar\psi^{a_n}$, we define a connected topological type $\Gamma$ with $n$ half-edges according to the following. 
\begin{itemize}
    \item $g=0$ and the curve class is $\beta$.
    \item If $i_l=0$, the $l$th half-edge is a leg.
    \item If $i_l\neq 0$, the $l$th half-edge is a root of weight $i_l$.
    \item $\rho$ is the number of nonzero elements in $i_1,\ldots,i_n$. 
    \item Say, $i_{r_1},\ldots,i_{r_\rho}$ are all the nonzero elements preserving orders. Then $\vec{\mu}=(i_{r_1},\ldots,i_{r_\rho})$.
\end{itemize}
\begin{defn}\label{defn:inv}
Under such notation, $I_\beta$ is defined to be a multilinear form over $\HH[\bar\psi]$ generated by the following equation.
\[
I_\beta(\bar\psi^{a_1}[\alpha_1]_{i_1}, \ldots, \bar\psi^{a_n}[\alpha_n]_{i_n}) = \displaystyle\int_{\mathfrak c_\Gamma(X/D)} \prod\limits_{l=1}^n\bar\psi^{a_l}_l\ev_l^*\alpha_l.
\]
If the condition
\begin{equation}\label{eqn:tangcond}
\sum_{l=1}^n i_l = \int_\beta D
\end{equation}
is not satisfied, then $I_\beta(\bar\psi^{a_1}[\alpha_1]_{i_1}, \ldots, \bar\psi^{a_n}[\alpha_n]_{i_n})$ is defined to be $0$.
\end{defn}
Here the evaluation map $ev_l$ corresponding to the $l$th marking lands on $D$ if $i_l\neq 0$ and on $X$ if $i_l=0$ (see Definition \ref{rel-inv-neg1} for more details).

By defining invariants this way, we get rid of graph notation and incorporate the contact order information into insertions. In this notation, TRR and WDVV equation can be stated in nice ways. For simplicity, we assume that only cohomology classes of even degrees are used.
\begin{prop}[TRR]\label{prop:TRR}
\begin{align*}
    &I_\beta(\bar\psi^{a_1+1}[\alpha_1]_{i_1}, \ldots, \bar\psi^{a_n}[\alpha_n]_{i_n}) \\
    =&\sum I_{\beta_1}(\bar\psi^{a_1}[\alpha_1]_{i_1}, \prod\limits_{j\in S_1} \bar\psi^{a_j}[\alpha_j]_{i_j}, \widetilde T_{i,k}) I_{\beta_2}(\widetilde T_{-i}^k, \bar\psi^{a_2}[\alpha_2]_{i_2}, \bar\psi^{a_3}[\alpha_3]_{i_3}, \prod\limits_{j\in S_2} \bar\psi^{a_j}[\alpha_j]_{i_j}),
\end{align*}
where the sum is over all $\beta_1+\beta_2=\beta$, all indices $i,k$ of basis, and $S_1, S_2$ disjoint sets with $S_1\cup S_2=\{4,\ldots,n\}$. Also, the $\prod$ symbol makes each factor as a separate insertion, instead of multiplying them up.
\end{prop}
\begin{prop}[WDVV]\label{prop:WDVV}
\begin{align*}
    &\sum I_{\beta_1}(\bar\psi^{a_1}[\alpha_1]_{i_1}, \bar\psi^{a_2}[\alpha_2]_{i_2}, \prod\limits_{j\in S_1} \bar\psi^{a_j}[\alpha_j]_{i_j}, \widetilde T_{i,k}) I_{\beta_2}(\widetilde T_{-i}^k, \bar\psi^{a_3}[\alpha_3]_{i_3}, \bar\psi^{a_4}[\alpha_4]_{i_4}, \prod\limits_{j\in S_2} \bar\psi^{a_j}[\alpha_j]_{i_j}) \\
    =&\sum I_{\beta_1}(\bar\psi^{a_1}[\alpha_1]_{i_1}, \bar\psi^{a_3}[\alpha_3]_{i_3}, \prod\limits_{j\in S_1} \bar\psi^{a_j}[\alpha_j]_{i_j}, \widetilde T_{i,k}) I_{\beta_2}(\widetilde T_{-i}^k, \bar\psi^{a_2}[\alpha_2]_{i_2}, \bar\psi^{a_4}[\alpha_4]_{i_4}, \prod\limits_{j\in S_2} \bar\psi^{a_j}[\alpha_j]_{i_j}),
\end{align*}
where each sum is over all $\beta_1+\beta_2=\beta$, all indices $i,k$ of basis, and $S_1, S_2$ disjoint sets with $S_1\cup S_2=\{5,\ldots,n\}$. Also, the $\prod$ symbol makes each factor as a separate insertion, instead of multiplying them up.
\end{prop}

We would like to note that the sums in TRR and WDVV are finite  because for a given $\beta_1$, when all insertions but one are fixed, the remaining $\widetilde T_{i,k}$ is subject to the virtual dimensional constraint in Proposition \ref{prop:vdim}. Besides TRR and WDVV, we have other well-known properties in Gromov--Witten theory.
\begin{prop}[String equation]
\begin{align*}
    &I_\beta([1]_0, \bar\psi^{a_1}[\alpha_1]_{i_1}, \ldots, \bar\psi^{a_n}[\alpha_n]_{i_n}) \\
    =&I_\beta(\bar\psi^{a_1-1}[\alpha_1]_{i_1}, \ldots, \bar\psi^{a_n}[\alpha_n]_{i_n}) + \ldots + I_\beta(\bar\psi^{a_1}[\alpha_1]_{i_1}, \ldots, \bar\psi^{a_n-1}[\alpha_n]_{i_n}).
\end{align*}
\end{prop}
\begin{prop}[Divisor equation]
Let $\omega\in H^2(X)$.
\begin{align*}
    &I_\beta([\omega]_0, \bar\psi^{a_1}[\alpha_1]_{i_1}, \ldots, \bar\psi^{a_n}[\alpha_n]_{i_n}) \\
    =&\Big(\int_\beta \omega\Big)  I_\beta(\bar\psi^{a_1}[\alpha_1]_{i_1}, \ldots, \bar\psi^{a_n}[\alpha_n]_{i_n}) + I_\beta(\bar\psi^{a_1-1}[\alpha_1\cdot\omega]_{i_1}, \ldots, \bar\psi^{a_n}[\alpha_n]_{i_n}) \\
    &+ \ldots + I_\beta(\bar\psi^{a_1}[\alpha_1]_{i_1}, \ldots, \bar\psi^{a_n-1}[\alpha_n\cdot\omega]_{i_n}).
\end{align*}
\end{prop}
\begin{prop}[Dilaton equation]
\begin{equation*}
    I_\beta(\bar\psi [1]_0, \bar\psi^{a_1}[\alpha_1]_{i_1}, \ldots, \bar\psi^{a_n}[\alpha_n]_{i_n}) = (n-2)I_\beta(\bar\psi^{a_1}[\alpha_1]_{i_1}, \ldots, \bar\psi^{a_n}[\alpha_n]_{i_n}).
\end{equation*}
\end{prop}
String, divisor, dilaton, TRR and WDVV equations are direct consequences of Definitions \ref{def:cycle} and \ref{rel-inv-neg1}.

\subsection{Relative quantum rings} \label{sec:quan}
Let $\mt=\sum t_{i,k} \widetilde T_{i,k}$ where $t_{i,k}$ are formal variables. For simplicity, we write the set of all formal variables as $\{t_{i,k}\}$. Note that $i$ takes values in $\Z$ and $k$ ranges over labels of basis in the corresponding cohomology rings. Also note that it is an infinite set.

If one considers odd-degree cohomology classes, we have to impose supercommutativity among $\{t_{i,k}\}$ as well as on the interactions between $\{t_{i,k}\}$ and $\{\widetilde T_{i,k}\}$. This is a standard procedure and we omit the details. If one does not consider odd-degree classes, most of the theories (except Virasoro constraints) can be built on the subring of even-degree classes.

We denote the Novikov variable by $q$ and form the Novikov ring $\C[\![\NE(X)]\!]$ where $\NE(X)$ is the cone of effective curve classes in $X$. Denote its formal power series ring with variables in $\{t_{i,k}\}$ (infinitely many) by
\[
\C[\![\NE(X)]\!][\![\{t_{i,k}\}]\!].
\]
In order to define the genus-zero potential function, we have to work on a completion of this ring. Define the ideals
\[
I_p=(\{t_{i,k}\}_{|i|\geq p})
\]
for $p\geq 0$. They form a chain
\[
I_0\supset I_1\supset I_2\supset \cdots\text{.}
\]
Note that $\bigcap\limits_{p\geq 0} I_p=0$. We now have the completion
\begin{equation}\label{eqn:completion}
\C[\![\NE(X)]\!]\widehat{[\![\{t_{i,k}\}]\!]}=\varprojlim \C[\![\NE(X)]\!][\![\{t_{i,k}\}]\!]/I_p.
\end{equation}

Define the \emph{genus-zero Gromov--Witten potential} to be
\begin{equation*}
\Phi(\mt) = \sum\limits_{n\geq 3}\sum\limits_{\beta} \dfrac{1}{n!}I_{\beta}( \underbrace{\mt,\ldots,\mt}_n )q^\beta \in \C[\![\NE(X)]\!]\widehat{[\![\{t_{i,k}\}]\!]}.
\end{equation*}
Explicitly, $\Phi$ is a formal function in variables $\{t_{i,k}\}$, and can be explicitly written as the following.
\begin{equation}\label{eqn:phigeneral}
\Phi(\ldots,t_{i,k},\ldots) = \sum\limits_{\{n_{i,k}\}}\sum\limits_\beta I_\beta(\prod\limits_{i,k}\widetilde T_{i,k}^{n_{i,k}})\prod\limits_{i,k}\dfrac{t_{i,k}^{n_{i,k}}}{n_{i,k}!}q^\beta,
\end{equation}
where the first sum ranges over sets of nonnegative integers $\{n_{i,k}\}$ such that all but finitely many $n_{i,k}$ are $0$. Also, $\prod\limits_{i,k}\widetilde T_{i,k}^{n_{i,k}}$ should be understood as putting $\sum n_{i,k}$ insertions, with $\widetilde T_{i,k}$ repeated $n_{i,k}$ times.

\begin{rmk}
There is a difference here between relative Gromov--Witten theory and absolute Gromov--Witten theory. In absolute Gromov--Witten theory, fixing a $q^\beta$ in $\Phi(\mt)$, its coefficient is a polynomial in their formal variables. In relative Gromov--Witten theory, the coefficient of $q^\beta$ in $\Phi(\mt)$ is no longer a polynomial. This is inevitable even for $q=0$. 
\end{rmk}

Now we define quantum product `$\star$' by the following rule.
\begin{equation}\label{eqn:qr}
\widetilde T_{i_1,k_1} \star \widetilde T_{i_2,k_2} = \sum\limits_{i_3,k_3} \dfrac{\partial^3 \Phi}{\partial t_{i_1,k_1} \partial t_{i_2,k_2} \partial t_{i_3,k_3}} \widetilde T_{-i_3}^{k_3}.
\end{equation}
Note that $\widetilde T_{-i_3}^{k_3}$ and $\widetilde T_{i_3,k_3}$ are dual under the pairing. This defines a ring structure on $\HH[\![\NE(X)]\!]\widehat{[\![\{t_{i,k}\}]\!]}$.

We can also define small relative quantum ring. In 
\begin{equation}\label{eqn:3derivative}
\dfrac{\partial^3 \Phi}{\partial t_{i_1,k_1} \partial t_{i_2,k_2} \partial t_{i_3,k_3}},
\end{equation}
we can set $t_{i,k}=0$ if $i\neq 0$ or $\widetilde T_{i,k}$ is not a degree-zero or (real) degree-two cohomology class. We denote small relative quantum product by $\smst$.

In absolute Gromov--Witten theory, if we set $q=0, \mt=0$ in quantum cohomology, we recover the classical product structure of cohomology ring. In relative Gromov--Witten theory, we have a similar result as follows.

\begin{ex}[Specialization at $q=0, \mt=0$] \label{ex:special}
Under the specialization at $q=0, \mt=0$, the definition of quantum product becomes
\[
\widetilde T_{i_1,k_1} \star_{q=0,\mt=0} \widetilde T_{i_2,k_2} = \sum\limits_{i_3,k_3} I_0(\widetilde T_{i_1,k_1}, \widetilde T_{i_2,k_2}, \widetilde T_{i_3,k_3}) \widetilde T_{-i_3}^{k_3}.
\]
We claim that 
\[I_0(\widetilde T_{i_1,k_1}, \widetilde T_{i_2,k_2}, \widetilde T_{i_3,k_3}) = A(\widetilde T_{i_1,k_1}, \widetilde T_{i_2,k_2}, \widetilde T_{i_3,k_3}),\] where $A$ is the tri-linear function defined in \eqref{eqn:trilinear}. Note if one of $i_1,i_2,i_3$ is negative, it is simply a rubber invariant over $D$ (no graphs of $\infty$ type due to trivial curve class). Furthermore, the rubber moduli is $\bM_{0,3}\times D$ and the invariant is simply an integration over $D$. When two of $i_1,i_2,i_3$ are negative, it is again a rubber invariant times an extra class $C_{\fG}$ (see \eqref{eqn:cg}). The moduli space is again $\bM_{0,3}\times D$, and one computes that $C_{\fG}$ is simply the divisor class of $D$. Thus, we have an extra $D$ factor in the integration.
\end{ex}
\begin{rmk}
There is yet another difference between absolute and relative quantum ring. In absolute theory, if we set $q=0$, those \eqref{eqn:3derivative} already become independent of $\mt$ (in other words, has the same effect as setting $\mt=0$) because a nontrivial degree-zero nondescendant invariant has to be a three-point invariant. In relative theory, if we set $q=0$, we could still have nontrivial $n$-point invariants with $n>3$. In this case, bipartite graphs simplifies to a single vertex of type $0$, and moduli space is simply a product $\bM_{0,n}\times D$. But $C_{\fG}$ could involve $\psi$-classes on $\bM_{0,n}$, resulting in nonzero integrals.
\end{rmk}

Relative quantum ring is in fact a bigraded ring. We define the two gradings as extensions of $\deg^{(1)}, \deg^{(2)}$ defined in \eqref{eqn:deg1}, \eqref{eqn:deg2}. Besides \eqref{eqn:deg1} and \eqref{eqn:deg2}, we further define
\[
\deg^{(1)}(q^\beta)=\int_\beta D, \quad \deg^{(1)}(t_{i,k})=-i,
\]
\[
\deg^{(2)}(q^\beta)=\int_\beta c_1(T_X(\mathrm{-log}D)), \quad \deg^{(2)}(t_{i,k})=1-\mathrm{deg}^{(2)}(\widetilde T_{i,k}).
\]
Relative quantum product preserves the first grading because of the condition \eqref{eqn:tangcond}. The second grading is preserved because of the virtual dimension computed in Proposition \ref{prop:vdim}.

\subsection{An example: small relative quantum ring of $(\Pp^n,\Pp^{n-1})$}
Let $H$ be the class of hyperplane in $\Pp^n$. And $q^d$ means that the curve class is of degree $d$. By divisor equation, genus-zero potential is a series in $qe^{t_{0,1}}$ where we assume that $\widetilde T_{0,1}=[H]_0$. As a typical convention in Gromov--Witten theory, we can simply set $t_{0,1}=0$ to cut down the unnecessary variable. 

First, we note that 
\[
\deg^{(1)}(q)=1,\quad \deg^{(2)}(q)=n.
\]
It is easy to compute that
\begin{equation}\label{eqn:relation}
    [1]_1 \smst [H^n]_0 = [1]_0q.
\end{equation}
We claim that the relation \eqref{eqn:relation}, plus the associativity, the two gradings, and the classical ring structure of $\HH$ determines the whole small relative quantum ring.

First, gradings already give us a lot of information. For example, because $\deg^{(2)}(q)=n$, we easily conclude that 
\[
[H^a]_i\smst [H^b]_j=[H^{a+b}]_{i+j}\qquad \text{if } a+b< n, i,j>0.\] 
For $[H]_0\smst [H^n]_0$, multiply both sides of \eqref{eqn:relation} by $[1]_{-1}$. We have \[
[1]_{-1}\smst [1]_1\smst [H^n]_0=[H]_0\smst [H^n]_0=[1]_{-1}q.\] 
This also allows us to compute $[H^a]_0\smst [H^{b}]_0$ with $a+b>n$ and a lot of others.

We are going to write this ring in a neater way. Since $[1]_0$ is the identity, we write it as $1$. We also set 
\[
[1]_1=x,\quad [H]_0=y.
\]
Thus, $[H^a]_0=y^a$ if $i\leq n$. $[1]_{-i}$ can be denoted as $y/x^i$ if $i>0$, but we keep in mind that $x$ itself is not invertible. Relation \eqref{eqn:relation} becomes $xy^n=q$. An arbitrary element of positive grading can be written as
\[
[H^a]_i=y^ax^i\qquad\text{for }a<n,\; i>0.\] 
An arbitrary element of negative grading can be written as
\[
[H^a]_{-i}=y^{a+1}/x^i\qquad\text{for }a<n,\; i>0.
\]
As a result, we conclude the following:
\begin{thm}
The small relative quantum ring of $(\Pp^n,\Pp^{n-1})$ is isomorphic to the sub-$\C$-algebra of \[\C[x,x^{-1},y,q]/(xy^n-q)\] generated by $1,x,y/x,y/x^2,\ldots$. 
\end{thm}

There is a nicer way to interpret the relationship between the quantum and the classical limit ring in terms of infinitely generated monoids. Denote the two generators of the abelian group $\Z^2$ by $e_1$ and $e_2$ and let $\sigma\subset \Z^2$ be the $2$-dimensional cone generated by $e_1, -e_1, e_2$. Let $\sigma'$ be the ray generated by $-e_1$. Then $P=\sigma\backslash\sigma'$ is an infinitely generated cone.
\begin{lem}
The small relative quantum ring is isomorphic to the algebra $\C[P]$ after identifying $e_1$ with $x$ and $e_2$ with $y$.
\end{lem}
Note that $q$ disappears because it is already generated by $x$ and $y$. Since $q=xy^n$, we get the classical limit defined in Section \ref{sec:ins} by setting $xy^n=0$. So we have following result:
\begin{lem}
The ring of insertions (classical limit of the quantum ring) of $(\Pp^n,\Pp^{n-1})$ is isomorphic to the quotient ring $\C[P]/(xy^n)$.
\end{lem}

\subsection{Givental formalism in genus zero}\label{sec:Givental}
A good reference on genus-zero Givental formalism could be \cite{Giv1}. Besides, a lot of other works contain good introductions to the Lagrangian cones including \cites{CCIT, CIJ, LPS, Lproceeding}, among others. 
In fact, \cite{Lproceeding} and \cite{LPS}*{Section 3} adopt an axiomatic approach which also applies to our situation. The key is that a right set-up needs to be given so that these equations can be organized as the same differential equations as \cite{Giv1}*{(DE),(SE),(TRR)}. The rest of the properties simply follow formally.
In this section, we describe the set-up and briefly recall Givental's formalism of the Lagrangian cones. But we do not repeat the details.

Let
\[
\mathcal H=\HH \otimes_\C \C[\![\NE(X)]\!](\!(z^{-1})\!),
\]
where $(\!(z^{-1})\!)$ means formal Laurent series in $z^{-1}$. It has a polarization
\[
\mathcal H=\mathcal H_+\oplus\mathcal H_-,
\]
where
\[
\mathcal H_+=\HH \otimes_\C \C[\![\NE(X)]\!][z], \quad \text{and} \quad \mathcal H_-=z^{-1}\HH \otimes_\C \C[\![\NE(X)]\!][\![z^{-1}]\!].
\]

There is a $\C[\![\NE(X)]\!]$-valued symplectic form
\[
\Omega(f,g)=\text{Res}_{z=0}(f(-z),g(z))dz,
\]
where the pairing $(f(-z),g(z))$ takes values in $ \C[\![\NE(X)]\!](\!(z^{-1})\!)$ and is induced by the pairing on $\HH$. There is a natural symplectic identification between $\mathcal H_+\oplus \mathcal H_-$ and the cotangent bundle $T^*\mathcal H_+$. 

To parametrize points in $\mathcal H$, we need more formal variables than in the previous subsection. For $l\geq 0$, we write $\mt_l=\sum\limits_{i,k} t_{l;i,k}\widetilde T_{i,k}$ where $t_{l;i,k}$ are formal variables. Also write
\[
\bt(z)=\sum\limits_{l=0}^\infty \mt_l z^l.
\]
\emph{The relative genus-zero descendant Gromov--Witten potential} is defined as
\[
\mathcal F(\bt(z))=\sum\limits_\beta \sum\limits_{n=0}^\infty \dfrac{q^\beta}{n!} I_\beta(\underbrace{\bt(\bar\psi),\ldots,\bt(\bar\psi)}_{n}).
\]
To compare with the notation in \cites{Giv1, Lproceeding}, note that their $H$ corresponds to our $\HH$, their parameters $t_l^\mu$ correspond to our $t_{l;i,k}$ (both parametrize classes from rings of insertions). We further let $\widetilde T_{0,0}=[1]_0$.
\begin{prop}
$\mathcal F(\bt)$ satisfies (DE),(SE),(TRR) in \cite{Giv1} (or equivalently, \cite{Lproceeding} with $G_0=\mathcal F$).
\end{prop}
Following \cite{Giv1}, we define the dilaton-shifted coordinates of $\mathcal H_+$
\[
\bq(z)=\mq_0+\mq_1z+\mq_2z^2+\ldots=-z+\mt_0+\mt_1z+\mt_2z^2+\ldots.
\]
Coordinates in $\mathcal H_-$ are usually chosen so that $\bq, \bp$ form Darboux coordinates. 
\[
\bp(z)=\mmp_0z^{-1}+\mmp_1z^{-2}+\ldots = \sum\limits_{l\leq -1} \sum\limits_{i,k} p_{l;i,k}\widetilde T_{-i}^k z^l.
\]
Givental's Lagrangian cone $\mathcal L$ is then defined as the graph of the differential $d\mathcal F$. More precisely, a (formal) point in the Lagrangian cone can be explicitly written as
\[
-z+\bt(z)+\sum\limits_{\beta} \sum\limits_{n} \sum\limits_{i,k} \dfrac{q^\beta}{n!} I_\beta\Big(\dfrac{\widetilde T_{i,k}}{-z-\bar\psi},\underbrace{\bt(\bar\psi),\ldots,\bt(\bar\psi)}_{n}\Big) \widetilde T_{-i}^k.
\]

\begin{rmk}[I-functions]
One might ask whether we have I-functions and mirror theorems in this story. In fact, in view of Definition \ref{def:cycle}, if $X$ is toric and $D$ is torus invariant, we can simply write out the I-function for the corresponding $r$-th root stack, and then take a suitable limit for $r$ to fit the function into our formalism. Following Definition \ref{def:cycle}, the procedures are very straightforward and we believe it is unnecessary to spell it out. When $(X,D)$ is not a toric pair, a mirror theorem for the pair $(X,D)$ has recently been proved in \cite{FTY} using our formalism.
\end{rmk}

\subsection{Virasoro constraints} \label{sec:Virasoro}
In absolute Gromov--Witten theory, Virasoro constraints have a long history. Early works include \cites{EHX,EJX,LiuT}, with a lot of other works we are not able to fully recall. In this section, we follow \cite{Giv1} and describe Virasoro constraints as quantized operators. Virasoro constraints automatically hold in genus zero due to the properties of the Lagrangian cone. However, commutators of operators may not be well defined in our case (see Remark \ref{rmk:badnews}). 

Recall that an operator $A$ is called \emph{infinitesimal symplectic} if it satisfies $\Omega(A(f),g)+\Omega(f,A(g))=0$. Given a class $[\alpha]_i\in \HH$ such that if $i=0$, $\alpha\in H^{p,q}(X)$, and if $i\neq 0$, $\alpha\in H^{p,q}(D)$. Define two operators $\rho, \mu$ as the following.
\begin{align*}
    \begin{split}
        \rho([\alpha]_i) &= [\alpha\cup c_1(T_X(-\mathrm{log} D))]_i,\\
        \mu([\alpha]_i) &= \begin{cases}
        [(\dim_\C(X)/2-p)\alpha]_i, \text{ if } i\geq 0,\\
        [(\dim_\C(X)/2-p-1)\alpha]_i, \text{ if } i<0.\\
        \end{cases}
    \end{split}
\end{align*}
The extra $-1$ in the definition of $\mu$ on negative parts is in fact consistent with the heuristic view in Remark \ref{rmk:heuristic}. One should also compare this definition with \cite{JT}, since relative theory is considered as a limit of orbifold theory according to Definition \ref{def:cycle}.

Now, we can construct the following transformations.
\begin{align*}
    \begin{split}
        &l_{-1}=z^{-1},\\
        &l_{0}=zd/dz + 1/2 + \mu + \rho/z,\\
        &l_m=l_0(zl_0)^m.
    \end{split}
\end{align*}
One can check that they are infinitesimal symplectic. On the other hand, one can check that these operators satisfies the following (similar to \cite{Giv1}*{Virasoro constraints}). 
\[
    \{l_m,l_n\}=(n-m)l_{m+n},
\]
where $\{l_m,l_n\}=l_ml_n-l_nl_m$ is the Poisson bracket.

In general, an infinitesimal symplectic transformation $A$ gives rise to a vector field on $\mathcal H$ in the following way. Given a point $f\in \mathcal H$, the tangent space of $\mathcal H$ at $f$ can be naturally identified with $\mathcal H$ itself. By assigning a point $f$ the vector $A(f)\in T_f\mathcal H$, we obtain a tangent vector field on $\mathcal H$. The following proposition shares almost exactly the same proof as \cite{Giv1}*{Theorem 6}. Thus, we do not repeat its argument here.
\begin{prop}
The vector field generated by each $l_m$ is tangent to the Lagrangian cone $\mathcal L$.
\end{prop}

With this proposition, one can easily argue that each $l_m$ is associated with a Hamiltonian function $\Omega(l_m f,f)/2$ on $\mathcal L$. To fully parametrize this function, recall we have infinitely many variables $\{t_{l;i,k}\}$. To make sense of this Hamiltonian function, it needs to sit inside the completion $\C[\![\NE(X)]\!]\widehat{[\![\{t_{i,k}\}]\!]}$ defined in \eqref{eqn:completion}. We can similarly define the quantization of quadratic Hamiltonian according to the following rules.
\begin{align*}
    (q_{l;i,k}q_{l';i',k'})^{\wedge}&=q_{l;i,k}q_{l';i',k'}/\hbar, \\ (q_{l;i,k}p_{l';i',k'})^{\wedge}&=q_{l;i,k}\partial/\partial q_{l';i',k'}, \\ (p_{l;i,k}p_{l';i',k'})^{\wedge}&=\hbar \partial^2/\partial q_{l;i,k}\partial q_{l';i',k'}.
\end{align*}
Thus, we obtain a sequence of quantized operators $L_m=\widehat{l_m}$. Although we have not yet defined higher genus relative invariants with negative relative markings, we can still look at the restriction of their actions on genus-zero part. More precisely, we have the following proposition.
\begin{prop}\label{prop:genus0vir}
For $m\geq -1$, we have the equality
\[
    [e^{-\mathcal F(\bt)/\hbar} L_m e^{\mathcal F(\bt)/\hbar}]_{\hbar^{-1}}=0,
\]
where $[\cdot]_{\hbar^{-1}}$ means taking the $\hbar^{-1}$ coefficient.
\end{prop}
This is a standard conclusion from the fact that $l_m$ being infinitesimal symplectic and the induced tangent vector field is tangent to the Lagrangian cone $\mathcal L$.

\begin{ex}
Here we explicitly compute $L_{-1}$ and $L_0$.
\begin{align*}
    L_{-1} &= -\dfrac{\partial}{\partial t_{0;0,1}} + \sum\limits_{l,i,k} t_{l+1;i,k} \dfrac{\partial}{\partial t_{l;i,k}} + \dfrac{1}{2}(\mt_0,\mt_0),\\
    L_0 &= -\dfrac{1}{2}(3-\dim)\dfrac{\partial}{\partial t_{1;0,0}} + \sum\limits_{l,i,k} (-\mu(\widetilde T_{i,k}) + l + \dfrac{1}{2})t_{l;i,k}\dfrac{\partial}{\partial t_{l;i,k}} \\
    &- \sum\limits_{k} ([c_1(T)]_0,\widetilde T_{0}^k)\dfrac{\partial}{\partial t_{0;0,k}}
    + \sum\limits_{i,k,i',k'} ([c_1(T)]_0 \cdot \widetilde T_{i,k},\widetilde T_{i'}^{k'}) t_{l+1;i,k} \dfrac{\partial}{\partial t_{l;-i',k'}} \\
    &+ \sum\limits_{i,k,i',k'}\dfrac{1}{2\hbar} ([c_1(T)]_0 \cdot \widetilde T_{i,k},\widetilde T_{i',k'}) t_{0;i,k}t_{0;i',k'},
\end{align*}
where we simplify some notation and write $\dim=\dim(X)$, $T=T_X(-\mathrm{log}D)$. As usual, here $i,k$ and $i',k'$ range over all possible basis. It is easy to check that the action of $L_{-1}$ is equivalent to string equation, and the action of $L_{0}$ reflects the fact that virtual dimension equals the degree of insertions.
\end{ex}

Although we already have $\{l_m,l_n\}=(n-m)l_{m+n}$ for unquantized transformations, we still hope to have similar Virasoro conditions on operators $L_m$. However, we might not be able to even compose two quantized operators due to infinite sum. To limit the number of variables, we set $\bq_i=0$ for $|i|>N$ where $N$ is a fixed integer. Given two infinitesimal symplectic transformations $A$ and $B$, we have
\[
\widehat A \widehat B - \widehat B \widehat A = \{A,B\}^{\wedge}+\mathcal C(h_A, h_B),
\]
where $h_A, h_B$ are the quadratic Hamiltonians associated with $A$ and $B$. Similar to \cite{JT}*{Section 3.1}, one computes that
\[
\mathcal C(p_{l;i,k}p_{l';i',k'}, q_{l;i,k}q_{l';i',k'}) = -\mathcal C( q_{l;i,k}q_{l';i',k'},p_{l;i,k}p_{l';i',k'}) = 1+\delta_{l,l'}\delta_{i,i'}\delta_{k,k'},
\]
and $\mathcal C=0$ for other quadratic monomials. Now, for the commutator $\widehat A \widehat B - \widehat B \widehat A$ to be well defined, we have to require the limit of $\mathcal C(h_A,h_B)$ under $N\to \infty$ exists. However, we have a bad news discussed in the following remark.

\begin{rmk}\label{rmk:badnews}
The commutator $L_{-1}L_1-L_1L_{-1}$ is in fact not well defined! We leave the details to readers. In fact, \cites{JT,CGT} already describe that to fix Virasoro condition on commutators, we need to add $\text{str}(1/4-\mu^2)/4$ to $L_0$. One can already see that this constant does not make sense in the obvious way because $\HH$ is infinite dimensional. Adding constants do not have an impact on Proposition \ref{prop:genus0vir}. But it would start to influence higher-genus theory. This is the reason why we still can not conjecture Virasoro operators for all-genus relative theory.
\end{rmk}

\appendix
\section{A lemma on Hurwitz--Hodge cycles} \label{sec:lemma-HHI}
Recall that $P_{D_0,r}$ is the root stack of $\Pp_D(L\oplus \sO)$, and $\cD_0, \cD_\infty$ are two invariant substacks. $\cD_0$ is isomorphic to $\sqrt[r]{D/L}$ and $\cD_\infty$ is isomorphic to $D$.
In this section, we show a formula computing some cycles on
$\bM_{0,\vec{a},n}(\cD_0)$ (we call them Hurwitz--Hodge cycles). Although we still write the target as $\cD_0$ in order to match notation, the set-up is in fact independent of $P_{D_0,r}$ and the lemma works for root gerbes in general.

Let $\bM^\sim_{\Gamma}(D)$ be the moduli of relative stable maps to rubber
targets over $D$. Suppose that $\Gamma$ is a rubber graph of one vertex whose $0$-roots
have weights $\vec{\mu}=(\mu_1,\ldots,\mu_{\rho_0})$, and $\infty$-roots have weights
$\vec{\nu}=(-\nu_1,\ldots,-\nu_{\rho_\infty})$ (recall that our convention in Section \ref{sec:graph} is to label negative numbers on $\infty$-roots). Also suppose that there are $n$ legs. Define
a vector of ages to be
\[
\vec{a}=(
(r-\nu_1)/r,\ldots,(r-\nu_{\rho_\infty})/r,\mu_1/r,\ldots,\mu_{\rho_0}/r,\underbrace{0,\ldots,0}_n
 ).
\]
We have the following forgetful
maps
\begin{align*}
  \begin{split}
  \tau_1:&\bM_{0,\vec{a}}(\cD_0,\beta)\rightarrow \bM_{0,n+\rho_0+\rho_\infty}(D,\beta),\\ 
  \tau_2:&\bM^\sim_{\Gamma}(D)\rightarrow \bM_{0,n+\rho_0+\rho_\infty}(D,\beta).
  \end{split}
\end{align*}
Here under $\tau_2$, $\infty$-roots are identified as markings $1,\ldots,\rho_\infty$, $0$-roots as markings $\rho_\infty+1,\ldots,\rho_\infty+\rho_0$, and legs as markings $\rho_\infty+\rho_0+1,\ldots,\rho_\infty+\rho_0+n$. For $1\leq i\leq \rho_\infty$, write 
\begin{equation*}
p_i=\nu_i\tau_2^*\psi_{i}-\ev_{i}^*c_1(L),
\end{equation*}
and 
\[
\sigma_l=\sum\limits_{1\leq i_1<\ldots<i_l\leq\rho_\infty}p_{i_1}\ldots p_{i_l}.
\]
We set $\sigma_l=0$ if $l>\rho_{\infty}$.

Recall $\Psi_\infty$ is the divisor corresponding to the cotangent line bundle determined by the relative divisor on $\infty$ side (see Section \ref{sec:rel-neg-set-up}).
On the root gerbe $\cD_0$, there is a universal line bundle $L_r$. On
$\bM_{0,\vec{a}}(\cD_0,\beta)$, there is the universal curve $\pi:\mathcal
C\rightarrow \bM_{0,\vec{a}}(\cD_0,\beta)$ and a map $f:\mathcal C\rightarrow
\cD_0$. Recall $\mathcal L_r= f^*L_r$ and $-R^*\pi_*\mathcal L_r=R^1\pi_*\mathcal L_r-R^0\pi_*\mathcal L_r \in K^0(\bM_{0,\vec{a}}(\cD_0,\beta))$. 
\begin{lem}\label{lemma:HHI}
In the above notation, for any positive integer $k$ and $r\gg 1$, we have the following identity.
\begin{align*}
  \begin{split}
    &r^{k+1}(\tau_1)_*\left( c_k(-R^*\pi_*\mathcal L_r)\cap [\bM_{0,\vec{a}}(\cD_0)]^{\vir}\right) \\
    =&(\tau_2)_*\left( (\Psi_\infty^k-\Psi_\infty^{k-1}\sigma_1+\ldots+(-1)^{k}\sigma_k) \cap [\bM^\sim_{\Gamma}(D)]^{\vir}\right).
  \end{split}
\end{align*}
\end{lem}
\begin{proof}
The proof runs by an application of results in \cite{ACW}. We match our situation with \cite{ACW} by identifying $X_r=P_{D_0,r}$, $D_r=\cD_0$. We have the diagram (under the notation of \cite{ACW})
\begin{equation*}
    \xymatrix{
& \bM^{\rel}(P_{D_0,r},\cD_0) \ar[ld]_\Psi \ar[rd]^\Phi & \\
\bM^{\rel}(P,D_0) & & \bM^{\text{orb}}(P_{D_0,r}).
}
\end{equation*}
Here, $\bM^{\rel}(P,D_0)$ is the moduli of relative stable maps (corresponding to our $\bM_\Gamma(P,D_0)$), and $\bM^{\text{orb}}(P_{D_0,r})$ is the moduli of orbifold stable maps (corresponding to our $\bM_{\Gamma}(P_{D_0,r})$). But we have not introduced a notation for their $\bM^{\rel}(P_{D_0,r},\cD_0)$ in this paper. Since there is no notational conflict, we will stick to the notation of \cite{ACW} for this proof.

Recall the following results of \cite{ACW}.
\[
\Psi_*([\bM^{\rel}(P_{D_0,r},\cD_0)]^{\vir}) = [\bM^{\rel}(P,D_0)]^{\vir},
\]
\[
\Phi_*([\bM^{\rel}(P_{D_0,r},\cD_0)]^{\vir}) = [\bM^{\text{orb}}(P_{D_0,r})]^{\vir}.
\]
There is a $\GM$ action on $P_{D_0,r}$ which is compatible with the scaling action on the fiber of $P$. This induces actions on all three moduli spaces. Thus we have equivariant virtual cycles for all of them. We first prove the following analog in equivariant setting.
\begin{lem}
Under the induced $\GM$ actions, we have
\[
\Psi_*([\bM^{\rel}(P_{D_0,r},\cD_0)]^{\vir}_{\GM}) = [\bM^{\rel}(P,D_0)]^{\vir}_{\GM},
\]
and
\[
\Phi_*([\bM^{\rel}(P_{D_0,r},\cD_0)]^{\vir}_{\GM}) = [\bM^{\text{orb}}(P_{D_0,r})]^{\vir}_{\GM}.
\]
\end{lem}
\begin{proof}
Let us go down to the definition of equivariant Chow groups. According to \cite{EG}*{Section 2.2}, the $i$th equivariant Chow group of a space $X$ under an algebraic group $G$ can be defined as follows. Let $V$ be a $l$-dimensional representation of $G$ with $U\subset V$ an equivariant open set where $G$ acts freely and whose complement has codimension more than $\dim(X)-i$. Then define
\[
A^G_i(X)=A_{i+l-g}((X\times U)/G),
\]
where $\dim(G)=g$. In our case, $G=\GM$. To compute $A^{\C}_i(X)$, we can choose $V=\C^N$ where $N$ is sufficiently large with $\GM$ acting by scaling, and $U=\C^N-\{0\}$. Now $(X\times U)/\GM$ is an $X$-fibration over $U/\GM\cong \Pp^{N-1}$. If $X$ is projective, it is easy to see that $(X\times U)/\GM$ is also projective.

We apply this construction to our situation. Let us first consider $(P\times U)/\GM$. It is easy to see that $D_0$ induces a divisor 
\[
(D_0\times U)/\GM\subset (P\times U)/\GM.
\]
Since $D_0$ is smooth, this divisor is also smooth. Denote the projection 
\[
\pi:(P\times U)/\GM\rightarrow \Pp^{N-1}.
\]
If we consider the Gromov--Witten theory of $(P\times U)/\GM$ relative to $(D_0\times U)/\GM$, and choose the curve class $\beta$ such that $\pi_*\beta=0$, the moduli space can be realized as a fibration over $\Pp^{N-1}$ as well:
\begin{equation}\label{eqn:rel1}
\bM^{\rel}((P\times U)/\GM,(D_0\times U)/\GM) \cong (\bM^{\rel}(P,D_0) \times U)/\GM.
\end{equation}
There are natural perfect obstruction theories on both sides and they are identified under this isomorphism (this uses the curve being genus zero). Similarly, we have
\begin{equation}\label{eqn:rel2}
\bM^{\rel}((P_{D_0,r}\times U)/\GM,(\cD_0\times U)/\GM) \cong (\bM^{\rel}(P_{D_0,r},\cD_0) \times U)/\GM.
\end{equation}
Since the matching of virtual classes in \cite{ACW} works for all smooth projective pairs without extra conditions, we conclude the virtual classes of left-hand sides of \eqref{eqn:rel1} and \eqref{eqn:rel2} match under pushforward. On the other hand, if $N$ is sufficiently large, the Chow groups of the right-hand sides of \eqref{eqn:rel1} and \eqref{eqn:rel2} are isomorphic to equivariant Chow groups of corresponding moduli spaces. Under this identification, one can directly check the construction of virtual classes and conclude that for each of \eqref{eqn:rel1} and \eqref{eqn:rel2}, virtual class of the left-hand side is identified with equivariant virtual class of the corresponding moduli. This argument concludes that $\Psi_*$ preserves the equivariant virtual classes. The argument for $\Phi_*$ is similar.
\end{proof}
Now that the lemma is established, we can apply the virtual localization theorem to see that $\Psi_*$ and $\Phi_*$ identify each localization residue. There are stabilization maps from both moduli to the moduli of stable maps, and thus forming the following commutative diagram.
\begin{equation*}
\xymatrix{
& \bM^{\rel}(P_{D_0,r},\cD_0) \ar[ld]_\Psi \ar[rd]^\Phi & \\
\bM^{\rel}(P,D_0) \ar[rd]^{\bar\Psi} & & \bM^{\text{orb}}(P_{D_0,r}) \ar[ld]_{\bar\Phi}\\
& \bM_{0,n+\rho}(P,\beta).   & 
}
\end{equation*}
The fact that $\Psi_*$ and $\Phi_*$ identify localization residues implies that corresponding localization residues of $\bM^{\rel}(P,D_0)$ and $\bM^{\text{orb}}(P_{D_0,r})$ agree under the pushforward of $\bar\Psi$ and $\bar\Phi$, respectively.

To conclude Lemma \ref{lemma:HHI}, consider the localization residue of $\bM^{\rel}(P,D_0)$ corresponding to a vertex of class $\beta$ supporting on the rubber over $D_0$, with $\rho_\infty$ edges going out of it of degrees $\nu_1,\ldots,\nu_{\rho_\infty}$. We also put $n$ interior markings and $\rho_0$ relative markings of profile $(\mu_1,\ldots,\mu_{\rho_0})$. The residue is
\[
\dfrac{\text{Edge}}{t-\Psi_\infty} \cap [\bM^\sim_{\Gamma}(D)]^{\vir}
\]
where $\text{Edge}$ is the edge contribution, and $t$ is the equivariant parameter. It corresponds to a similar graph in $\bM^{\text{orb}}(P_{D_0,r})$. One can also compute its residue, which is the following.
\[
\left(\sum_{j\geq 0} c_j(-R^*\pi_*\mathcal L_r)(t/r)^{\rho_\infty-1-j}\right)\prod\limits_{i=1}^{\rho_\infty}\dfrac{r}{t-\nu_i\bar \psi_i+ev_i^*c_1(L)} \text{Edge} \cap [\bM_{0,\vec{a}}(\cD_0,\beta)]^{\vir}.
\]
It is important to note that the edge contribution in orbifold case is the same as the one in relative case. Now push both localization residues to the corresponding fixed component of $\bM_{0,n+\rho}(P,\beta)$ (under our setting, $\rho=\rho_0$). Note that $\prod\limits_{i=1}^{\rho_\infty}\dfrac{r}{t-\nu_i\bar \psi_i+ev_i^*c_1(L)}$ and the edge contribution are in fact pullback classes from the moduli space of stable map. One can use projection formula and invert these factors (then edge contributions are cancelled) to solve for Chern classes. Lemma \ref{lemma:HHI} follows by comparing each coefficient of monomials in $t$.
\end{proof}

\begin{rmk}\label{rmk:rubcycle}
In fact, if one takes $k=0$ and uses the result in \cite{JPPZ18}, one concludes that \[
(\tau_2)_*[\bM^\sim_\Gamma(D)]^{\vir}=[\bM_{0,n+\rho_0+\rho_\infty}(D,\beta)]^{\vir}.
\]
So, the genus-zero hypersurface theory already determines a large part of the genus-zero rubber theory. Plus some relations of psi-classes and boundary classes, it is possible to write a version of definition which uses moduli of stable maps to $D$ instead of rubber moduli over $D$.
\end{rmk}

\newpage
\bibliographystyle{amsxport}
\bibliography{universal-BIB}
\end{document}